\documentclass[10pt, reqno]{amsart}
\usepackage{amsmath, amsthm, amscd, amsfonts, amssymb}

\usepackage{dsfont}
\usepackage[fleqn,tbtags]{mathtools}
\usepackage{graphicx}
\usepackage{color}
\usepackage[colorlinks]{hyperref}
\usepackage[all]{xy}
\usepackage[most]{tcolorbox}
\usepackage{tikz-opm}
\usepackage{tikz-cd}

\usepackage{dsfont}

\newcommand{\Rel}{{\bf Rel}}

\title{The category {\bf Rel}({\bf Nom}) }
\author[N. S. Razmara, M. Haddadi, and Kh. Keshvardoost]{N. S. Razmara, M. Haddadi$^1$, and Kh. Keshvardoost}
\address{{\bf M. Haddadi: }\rm{Faculty of Mathematics, Statistics and Computer Sciences, Semnan University, Semnan, Iran.}}
\email{ m.haddadi@semnan.ac.ir, haddadi$\_$1360@yahoo.com}

\address{{\bf Kh. Keshvardoosti: }\rm{Faculty of Mathematics, Statistics and Computer sciences, Velayat University, Iranshahr, Sistan and Baluchestan, Iran.}}
\email{khadijeh.keshvardoost@gmail.com, kh.keshvardoost@velayat.ac.ir}

\address{{\bf N. S. Razmara: }\rm{Faculty of Mathematics, Statistics and Computer Sciences, Semnan University, Semnan, Iran.}}
\email{ razmara@semnan.ac.ir, nssr$\_$29@yahoo.com}

\date{}
\begin{document}
	%-------------------------------------------------
	\newtheorem{theorem}{Theorem}[section]
	\newtheorem{lemma}[theorem]{Lemma}
	\newtheorem{proposition}[theorem]{Proposition}
	\newtheorem{corollary}[theorem]{Corollary}
	\theoremstyle{definition}
	\newtheorem{definition}[theorem]{Definition}
	\newtheorem{example}[theorem]{Example}
	\newtheorem{exercise}[theorem]{Exercise}
	\theoremstyle{remark}
	\newtheorem{remark}[theorem]{Remark}
	\newtheorem{note}[theorem]{Note}
	\newtheorem{notation}[theorem]{Notation}
	\newtheorem{My Guess}[theorem]{My Guess}
	%\numberwithin{equation}{section}
	%\newproof{proof}{Proof}
	%-------------------------------------------------

	\footnotetext[1] { Corresponding author}

	\begin{abstract}

	The category ${\rm Rel}(\mathcal{C})$ may be formed for any category $\mathcal{C}$ with finite limits using the same objects as $\mathcal{C}$ but whose morphisms from $X$ to $Y$ are binary relations in $\mathcal{C}$, that is, subobjects of $X\times Y$. In this paper, concerning the topos ${\bf Nom}$, we study the category ${\bf Rel}({\bf Nom})$. In this category, we define and investigate certain morphisms, such as deterministic morphisms.
	Then, stochastic mappings between nominal sets are defined by exploiting the underlying relation of functions between nominal sets. This allows one to reinterpret concepts and earlier results in terms of morphisms.

\end{abstract}
\maketitle

\section{Introduction and Preliminaries}

Finitely supported mathematics (or theory of nominal sets, when dealing with computer science applications) provides a framework for working with infinitely  structured hierarchically constructed by involving some basic elements (called atoms) by dealing only with a finite number of entities that form their supports, see \cite{Aleex}. This theory is related to the recent development of Fraenkel-Mostowski’s set theory, which works with ``nominal sets"
and deals with binding and new names in computer science, and developing by studying the category of nominal sets and equivariant functions between them, see \cite{Pitts2}. But some very common mathematical structures are not functions. Therefore, in this paper,  we introduce the category ${\bf Rel}({\bf Nom})$ consisting of nominal sets and equivariant relations between them which can have several advantages and is more expressive than the category of nominal sets alone, as it allows one to reason about relations between elements, not just the elements themselves. When working in ${\bf Rel}({\bf Nom})$, it is possible to reason about how permutations work  on the elements of the sets and the relations between them, which can be useful in fields such as physics and computer science.   In type theory, the category ${\bf Rel}({\bf Nom})$ can be used to model dependent types, which are types that depend on values, not just other types. This allows one to reason about the properties of programs that depend on input data. Also, the category ${\bf Rel}({\bf Nom})$ can be used to represent mathematical structures such as algebraic data types and reasoning about them in an equivalent way.

 Although the category ${\bf Rel}({\bf Nom})$ is not a topos, see Remark \ref{topos}, but presheaf  representation of nominal sets in ${\bf Rel}({\bf Nom})$ allows one to understand the mathematical structure of these objects in a more general and abstract way. Additionally, studying the equivariant relations between different presheaves can provide insight into the permutations and invariances of the sets being studied. Furthermore, in the field of computer science, in particular domain-specific languages, nominal sets and their presheaf representation can be used to reason about the syntax and semantics of programming languages in a more formal and rigorous way. Therefore, we devote Section 2 to introduce the category of ${\bf Rel}({\bf Nom})$ and explore some properties of its  morphisms. We then, in Section 3, discuss the presheaf representation of the objects of ${\bf Rel}({\bf Nom})$. Finally, in Section 4, we introduce another presheaf representation of nominal sets in ${\bf Rel}({\bf Nom})$ and introduce the natural deterministic and stochastic morphisms in this category.

% % % % % % % % % % % % % % % % % % % % % % % % % % % % % % % %

\subsection{The category $G$-\bf Set}\label{plmn}

This subsection is devoted to the needed facts about $G$-sets. We refer  interested readers to  \cite{dr.ebrahimi} and \cite{kilp}  for more information.

The set $X$ is equipped with a map 
$G\times{X}\longrightarrow{X}$ (action of the group $G$ on $X$) mapping $(g,x)$ to $gx$ called a $G$-\emph{set} if for every $g_{1}, g_{2}\in G$ and every $x\in{X}$, we have $g_1 {(g_{2}{x})}=(g_{1} g_{2}){x}$ and $e{x}=x$, in which   ``$e$" is the identity of the group $G$.
For $G$-sets $X$ and $Y$, a map $f:X\rightarrow Y$ is called an \emph {equivariant map} if $f(g {x}) =g{f(x)}$, for all $x\in X$ and $g\in G$.
 The category of all $G$-sets with equivariant maps between them denoted by $G$-$\bf Set$.

An element $x$ of a $G$-set $X$ is called a \emph {zero} (or a \emph {fixed}) element if $g x=x$, for all $g \in G$. We denote the set of all zero elements of a $G$-set $X$ by $\mathcal{Z}(X)$.

The $G$-set $X$ all of whose elements are zero is called \emph {discrete}, or a $G$-set with the \emph {identity action}.

A subset $Y$ of a $G$-set $X$ is an \emph {equivariant subset} (or a $G$-\emph{subset}) of $Y$ if for all $g\in G$ and $y\in Y$ we have
$g y\in Y$. The subset $\mathcal{Z}(X)$ of $X$ is a $G$-subset.

Given a $G$-set $X$ and $x\in X$, the set $Gx=\{g x: g \in G \}$ is called the \emph {orbit} of $x$. Note that the class  $\{Gx\}_{x\in X}$ is the corresponding partition of the  equivalence relation  $\sim$ over $X$ defined by $x\sim x'$ if and only if there exists $g \in G$ with  $g x=x'$, for which the class $x/\mathord\sim$ is denoted by $\mathsf{orb}x$. 

Given a $G$-set $X$ and $x\in X$, the set $G_{_{x}}=\{g\in G: gx=x\}$ is a subgroup of $G$ fixes $x$.
% % % % % % % % % % % % % % %
\subsection{The category of nominal sets}

In this subsection, we briefly recall relevant definitions concerning nominal sets. For the most part, we follow \cite{Pitts2}. 

From now on $\mathbb{D}$ denotes  a fixed, countably infinite set whose elements $a, b, c,\dots$ are called \emph{atomic names}. A \emph{permutation} $\pi$ of $\mathbb{D}$ is a bijective map from $\mathbb{D}$ to itself. All permutations of $\mathbb{D}$ with the composition of maps as the binary operation form a group called
the \emph {symmetric group} on $\mathbb{D}$ and denoted by ${\rm Sym ({\mathbb{D}}})$. 
 
A permutation $\pi\in{\rm Sym}(\mathbb{D})$ is \emph{finitary} if the set $\lbrace{d}\in{\mathbb{D}}: \pi{d}\neq{d}\rbrace$ is a finite subset of $\mathbb{D}$.  It is clear that the set ${\rm Perm}(\mathbb{D})$ consists of all finitary permutations is a subgroup of ${\rm Sym}(\mathbb{D})$. 

 Let $X$ be a set equipped with an action of the group  ${\rm Perm}(\mathbb{D})$, ${\rm Perm}({\mathbb{D}})\times{X}\longrightarrow{X}$ mapping $(\pi,x) \rightsquigarrow \pi {x}$. By definition of action of the group ${\rm Perm}(\mathbb{D})$ over the set $X$,  we have:
\begin{itemize}
\item[(i)] ${\pi_{1}}{(\pi_{2} {x})}=(\pi_{1}\circ \pi_{2}){x}$
\item[(ii)] $id {x}=x$,
\end{itemize}
 for every $\pi_{1}, \pi_{2}\in{{\rm Perm}({\mathbb{D}}})$ and every $x\in{X}$.

The set $\mathbb{D}$ together with the specified action given in Example \ref{example nom}(i) provide the most natural example of a ${\rm Perm}(\mathbb{D})$-set. In this case, for a given $C\subseteq \mathbb{D}$, using the notation $G_{_x}$ given in Subsection 1.1, we have: $$({\rm Perm}(\mathbb{D}))_{_{C}}=\{\pi \in {\rm Perm}(\mathbb{D}): \pi (d)=d, \ \forall d\in C\}.$$

Given a ${\rm Perm}(\mathbb{D})$-set $X$, a set of atomic names $C\subseteq\mathbb{D}$ is a \emph {support} for an element $x\in{X}$ if for all $\pi\in{{\rm Perm}(\mathbb{D}})$, we have:

  $$[\forall d\in C\, : \pi (d)=d)] \Longrightarrow \pi x=x.$$ In other words, $$\pi \in ({\rm Perm}(\mathbb{D}))_{_{C}}\Longrightarrow \pi x=x.$$
Given a ${\rm Perm}(\mathbb{D})$-set $X$, we say an element $x\in{X}$ is \emph{finitely supported} if $x$ has a finite support.

\begin{definition}\cite{Pitts2}
A \emph{nominal set} is a ${\rm Perm}(\mathbb{D})$-set, each of which element is finitely supported. 

Nominal sets are the objects of a category, denoted by $ \bf{Nom}$, whose morphisms are equivariant maps and whose composition and identities are as in the category of ${\rm Perm}(\mathbb{D})$-$\bf{Set}$. The category  $\bf{Nom}$ is a full subcategory of the category ${\rm Perm}(\mathbb{D})$-$\bf{Set}$.
\end{definition}

 \begin{remark}\label{supp for G}\cite[Propositions 2.1, 2.3]{Pitts2}
 Suppose $X$ is a nominal set and $x\in X$. 
 
 {\rm(i)} A finite subset $C\subseteq \mathbb{D}$ supports $x$ if and only if $(d_1\ d_2) x=x$, for all $d_1, d_2\notin C$.
 
 {\rm(ii)} Intersection of two finite supports of $x$ is a support of $x$.
 
 {\rm(iii)} By (ii), $x$ has the least finite support and is denoted by ${\rm supp}\,x$. In fact,
 ${\rm supp}_{_{X}}\,x=\bigcap\{C: C\  \text{is a finite support of}\, x\}$.
 \end{remark}

\begin{lemma}\cite{Pitts2}\label{Xfs is nom}
 If $X$ is a ${\rm Perm}({\mathbb{D}})$-set, then the subset $$X_{\rm fs}=\{x\in{X}: x\, \text{is finitely supported in  X}\}$$ of $X$, consisting of all finitely supported elements of $X$, is a nominal set.
 \end{lemma}
 
\begin{remark}\cite{Pitts2}\label{P(X)-is-act}
 {\rm (i)} Given a ${{\rm Perm}(\mathbb{D}})$-set $X$, the set $\mathcal{P}(X)=\{Y:{Y\subseteq{X}}\}$ with the following action 
 
 $${\rm Prem}(\mathbb{D})\times{\mathcal{P}(X)}\longrightarrow{\mathcal{P}(X)}$$
 $$(\pi ,Y)\rightsquigarrow{\pi\cdot{Y}=\lbrace {\pi{y}:{ y\in{Y} }}\rbrace}$$
 
  is a ${\rm Perm}(\mathbb{D})$-set. A set of atomic names $C$ supports
 $Y \in{\mathcal{P}(X)}$ if and only if
 $$(\forall {\pi\in{{\rm Perm}(\mathbb{D}}}))((\forall{d \in{ C}}) ~\pi(d) = d) \Longrightarrow{ (\forall ~y \in{Y} )~~ \pi{y} \in{Y}}.$$

{\rm (ii)} The equivariant subsets of $X$ are exactly the zero elements of $\mathcal{P}(X)$. Hence, we have ${\rm supp}\,Y=\emptyset$  if and only if $Y$ is an equivariant subset of $X$, for every $Y\in \mathcal{P}(X)$. Particularly, $X$ is supported by the empty set in $\mathcal{P}(X)$.  

\medskip
 
{\rm (iii)} The finitely supported elements of $\mathcal{P}(\mathbb{D})$ are finite and cofinite subsets of $X$; more explicitly, $C\in \mathcal{P}(\mathbb{D})$ is finitely supported if either $C$ or $\mathbb{D}- C$ is finite.
\end{remark}

In the following, we give some examples of nominal sets.
\begin{example}\label{example nom}
{\rm(i)} The set $\mathbb{D}$ is a nominal set,  with the action 
$${\rm Perm}(\mathbb{D})\times{\mathbb{D}}\longrightarrow{\mathbb{D}} $$
$$ (\pi,d)\leadsto{\pi(d)}.$$
Indeed, the set $\{d\}$ is a finite support of $d$, for every $d\in \mathbb{D}$.
\medskip

{\rm(ii)} Every discrete ${\rm Perm}(\mathbb{D})$-set $X$ is a nominal set. Indeed, the empty set is a finite support for each element $x\in X$. 
\medskip

{\rm(iii)} Each finite element of $\mathcal{P}(\mathbb{D})$ is supported by itself. So we get the nominal set $\mathcal{P}_{\rm f}(\mathbb{D})$ of all finite subsets of $\mathbb{D}$ 
with $\pi\cdot C = \{\pi d : d \in C \}$ and ${\rm supp}\, C= C$. 
\end{example}

\begin{remark}\label{supp is equi.}\cite[Proposition 2.11]{Pitts2}
If $X$ is a nominal set and $x\in X$, then $\pi {\rm supp}\,x={\rm supp}\, \pi x$, for every $\pi \in {\rm Perm}(\mathbb{D})$.
\end{remark}

Before moving on to the next lemma, we recall the following definition from \cite{kilp}.

\begin{definition}
A nominal set $X$ is called  

{\rm(i)} \emph{decomposable} if there exist  non-empty  nominal subsets $X_{1},X_{2}$, such that $X=X_{1}\cup X_{2}$ and $X_{1}\cap X_{2}=\emptyset$.  Otherwise, $X$ is \emph {indecomposable}.

{\rm(ii)} \emph{cyclic} if it is generated by only one element. That means it is of the form ${\rm Perm}(\mathbb{D})x$, for some $x\in X$.
\end{definition}

\begin{lemma}\label{cyclic nom}
If $X$ is a non-trivial indecomposable nominal set, then $X$ is cyclic and has no non-trivial nominal subset.
\end{lemma}
\begin{proof}
Let $x\in X$. Then, ${\rm Perm}(\mathbb{D})x\subseteq X$. If $X\neq {\rm Perm}(\mathbb{D})x$, then $X={\rm Perm}(\mathbb{D})x \cup (X \setminus{\rm Perm}(\mathbb{D})x)$ which is a contradiction. Note that, $X\setminus{\rm Perm}(\mathbb{D})x$ is an equivariant subset of $X$. Now, suppose $A$ is an equivariant subset of $X={\rm Perm}(\mathbb{D})x$. Let $a\in A$. Then, $a\in {\rm Perm}(\mathbb{D})x$ and so there exists $\pi \in {\rm Perm}(\mathbb{D})$ with $a=\pi x$. Thus, $x\in A$ and so $X=A$.
\end{proof}

  \begin{lemma}\label{suppor of fs map}
  Suppose $X$ and $Y$ are two nominal sets. Also, suppose $X'\in \mathcal{P}_{_{\rm fs}}(X)$ and $Y'\in \mathcal{P}_{_{\rm fs}}(Y)$. If $f:X'\longrightarrow Y'$ is a finitely supported map, then ${\rm supp}\,f(X')\subseteq {\rm supp}\,f \cup {\rm supp}\,X'$. Furthermore, ${\rm supp}\,f(x)\subseteq {\rm supp}\,f \cup {\rm supp}\,x$ when $X'=\{x\}$.
  \end{lemma}
  
  \begin{proof}
  Let $d_1, d_2\notin {\rm supp}\,f \cup {\rm supp}\,X'$. Then, $(d_1\ d_2) f=f$ and $(d_1\ d_2) X'=X'$. Since $(d_1\ d_2) f=f$, we have $f((d_1\ d_2) x)=(d_1\ d_2) f(x)$, for all $x\in X'$.
  Let $f(x)\in f(X')$ with $x\in X'$. Then, $(d_1\ d_2) X'=X'$ implies that $(d_1\ d_2) x\in X'$ and so $(d_1\ d_2) f(x)=f((d_1\ d_2)x)\in f(X')$.
     Thus, $(d_1\ d_2) f(X')=f((d_1\ d_2)X')$.
  \end{proof}
  
% % % % % % % % %
\subsection{The category $\Rel$}\quad
In this subsection, we review some elementary facts concerning the category of sets and relations, denoted by {\bf Rel}, from \cite{Barr}. The category {\bf Rel} is a category whose objects are sets and morphisms are relations, $R\subseteq X\times Y$. Here, the set of relations from $X$ to $Y$ is denoted by $\mathcal{R}(X,Y)$. By $R:X\longrightarrow Y$, we mean $R\in \mathcal{R}(X,Y)$. The composition of morphisms
$R \in \mathcal{R}(X,Y)$ and $S\in \mathcal{R}(Y, Z)$ is the relational composition $(S\circ R )\in \mathcal{R}(X, Z)$, defined by
$$(x,z) \in  (S\circ R ) \Longleftrightarrow \exists y\in Y; (x,y) \in  R  \text { and }  (y,z) \in S.$$
The identity morphism $id_X :X\to X$ is the identity relation $\Delta_X = \{(x,x):x\in X\}$. 
The category of $\bf{Set}$ is a full subcategory of $\bf{Rel}$.

\begin{definition}\cite{Barr}\label{def image and domain}
Suppose $R\in \mathcal{R}(X, Y)$. 

{\rm(i)} For given $S \subseteq X$, the set ${\overrightarrow R}(S)=\{y\in{Y}: \exists{x\in{S}}; (x,y)\in{R}\}$ is called the {\it direct image} of $S$ under $R$. Particularly, ${\overrightarrow R}(X)$ is called the {\it image} of $R$ and is denoted by ${\rm Im}R$. For the singleton subset $\{x\}\subseteq X$, the set ${\overrightarrow R}(\{x\})$ is denoted by ${\overrightarrow R}(x)$.

{\rm(ii)} For given $T \subseteq Y$, the set ${\overleftarrow R}(T)=\{x\in{X}: \exists{y\in{T}}; (x,y)\in{R}\}$ is called the {\it inverse image} of $T$ under $R$. Particularly, ${\overleftarrow R}(Y)$ is called the {\it domain} of $R$ and is denoted by ${\rm Dom}R$.
\end{definition}

\begin{definition}\cite{Barr}\label{zzzz}
A relation $R\in \mathcal{R}(X, Y)$ is said to be

(i) \emph{injective}  if $(x,y)\in R$ and $(x',y)\in R$ implies that $x=x'$.

(ii) \emph{surjective}  if for every $y\in Y$ there is some $x\in X$ so that $(x,y)\in R$.

(iii) \emph{total injective} if for every $x\in X$ there exists $y\in Y$ such that $x$ is the only element related to $y$. That is, if $(x,y)\in R$ and $(x',y)\in R$, then $x=x'$.

(iv) \emph{partial surjective map} if for every $y\in Y$ there exists $x\in X$ such that $y$ is the only element related to $x$. That is,  if $(x,y)\in R$ and $(x,y')\in R$, then $y=y'$.

(v) \emph{monic} if it is left cancelable; that is $R\circ S=R\circ T$ implies $S=T$. 

(vi)  \emph{epic} if it is right cancelable; that is $S\circ R=T\circ R$ implies $S=T$.

(vii)  \emph{well-defined} if $(x,y)\in R$ and $(x,y')\in R$ implies that $y=y'$.

\end{definition}

\begin{remark}\label{prop of BB}
If $R\in \mathcal{R}(X, Y)$ is a partial surjective map, then $\overleftarrow{R}(A)\cap \overleftarrow{R}(B)=\overleftarrow{R}(A\cap B)$.
To prove the non-trivial part, if $x\in \overleftarrow{R}(A)\cap \overleftarrow{R}(B)$, then there exist $a\in A$ and $b\in B$ with $(x,a), (x,b)\in R$. Now, since $R$ is a partial surjective map, we get that $a=b\in A\cap B$ and so $x\in \overleftarrow{R}(A\cap B)$. 
\end{remark}

One can easily prove the following lemma.

\begin{lemma}\label{lem for injec}
Suppose $X$ and $Y$ are two sets and $R\in \mathcal{R}(X , Y)$ is a relation. 

{\rm(i)} Let $R$ be injective. Then,
 $S=S'$, if  ${\overrightarrow R}(S)={\overrightarrow R}(S')$, for every $S,S'\subseteq {\rm Dom}R$.

 {\rm(ii)} Let $R$ be a partial map. Then,
 $T=T'$, if  ${\overleftarrow R}(T)={\overleftarrow R}(T')$, for every $T, T'\subseteq {\rm Im}R$.

\end{lemma}

\begin{proof}
{\rm(i)} Let $x\in S$. Then, there exists $y\in Y$ with $(x,y)\in R$ and so $y\in {\overrightarrow R}(S)=\overrightarrow{R}(S')$. Thus, $y\in \overrightarrow{R}(S')$ and so there exists $x'\in S'$ with $(x',y)\in R$. Now, since $R$ is injective, we get that $x=x'$. So, $x\in S'$. Similarly, we get that $S'\subseteq S$. 

{\rm(ii)} Let $y\in T$. Then, there exists $x\in X$ with $(x,y)\in R$ and so $x\in {\overleftarrow R}(T)=\overrightarrow{R}(T')$. Thus, $x\in \overleftarrow{R}(T')$ and so there exists $y'\in T'$ with $(x, y')\in R$. Now, since $R$ is a partial map, we get that $y=y'$. So, $y\in T'$. Similarly, we get that $T'\subseteq T$.
\end{proof}

 \begin{definition}\cite{Barr}\label{def of inverse}
 Let $X$, $Y$ be two sets and $R\in \mathcal{R}(X, Y)$. Then,

{\rm (i)} $R$ has a right inverse if there exists $S\in \mathcal{R}(Y, X)$ with $R\circ S=id_{_{{\rm Dom}S}}$.

{\rm(ii)} $R$ has a left inverse if there exists $S\in \mathcal{R}(Y, X)$ with $S\circ R=id_{_{{\rm Dom}R}}$.
 \end{definition}

\begin{remark}\label{inverse}
Let $X$, $Y$ be two sets and $R\in \mathcal{R}(X , Y)$. Then,

{\rm (i)} if $R$ has a right inverse, then $R$ is surjective.

{\rm(ii)} if $R$ has a left inverse, then $R$ is injective.
\end{remark}

We also recall from \cite{Anna} that the 
forgetful functor $F:\bf Set\longrightarrow \bf Rel$, which is the identity on objects and takes each map $f:X\longrightarrow Y$ to its underlying relation $\{(x,y)\in X\times Y :f(x)=y\}$, is a left adjoint for the powerset (or image) functor $P:\bf Rel\longrightarrow \bf Set$. This adjunction induces covariant powerset monad on $\bf{Set}$. $\bf{Rel}$ is isomorphic to the Kleisli category for this monad.

% % % % % % % % % % % % % % % % % % % % % % % % % % %

\section{The category {\bf Rel}({\bf Nom})}

In this section, we focus on the equivariant relations between the nominal sets rather than the equivariant functions between them and take into consideration the category {\bf Rel}({\bf Nom}). See the definition that follows.

 \begin{definition}\label{starr} 
 For given $G$-sets $X$ and $Y$, the set $\mathcal{R}(X,Y)$ is equipped with the action \label{starr}
$$\cdot :G \times \mathcal{R}(X,Y)\rightarrow \mathcal{R}(X,Y), \ \ \  g\cdot R=\{(g x, g y): (x,y)\in R\},$$ is a $G$-set.

\end{definition}

\begin{remark}\label{popol}
Given the nominal sets $X$ and $Y$,

{\rm (i)} 
a finite set $A\subseteq \mathbb{D}$ is a finite support for $R\in \mathcal{R}(X,Y)$  whenever,
$$\begin{array}{rcl}
\pi\in ({\rm Perm}(\mathbb{D}))_{_{A}} &\Longrightarrow & \pi \cdot R=R\\ 
&\Longrightarrow & R(x)=\pi (\overrightarrow{R}({\pi}^{-1} x)),
\end{array}$$
for every $x\in {\rm Dom}R$.

So, a relation $R:X\longrightarrow Y$ is equivariant if $\pi \cdot R=R$, for every $\pi\in {\rm Perm}(\mathbb{D})$.

{\rm (ii)} if $A$ is a finite support for the relation $R:X\longrightarrow Y$, then $\pi A$ is a finite support of $\pi\cdot R$.

{\rm (iii)} Since, by  Remark \ref{P(X)-is-act}(i), $\mathcal{R}(X,Y)=\mathcal{P}(X\times Y)$ is a ${{\rm Perm}(\mathbb{D}})$-set. Using  Lemma \ref{Xfs is nom}, the set of all finitely supported relations from $X$ to $Y$, denoted by $\mathcal{R}_{\rm fs}(X,Y)$, is a nominal set.
\end{remark}

 \begin{definition}
  Suppose $X$ and $Y$ are two nominal sets ($G$-sets) and $R\in \mathcal{R}(X,Y)$. The relation $R$ is \emph {equivariant} if it is an equivariant subset of $X\times Y$.
 \end{definition}

Now, we give some examples of equivariant relations.

  \begin{example}\label{suprel}
 Given nominal sets $X$ and $Y$, 
 
 {\rm(i)} the relation $\{(x,A)\in X\times \mathcal{P}_{_{\rm f}}(\mathbb{D}): {\rm supp}\,_{_{X}}x \subseteq A\}$, denoted by \linebreak ${\rm inc}:X\longrightarrow \mathcal{P}_{_{\rm f}}(\mathbb{D})$, is an equivariant element of $X\longrightarrow \mathcal{P}_{_{\rm f}}(\mathbb{D})$.
 
 \medskip
 
 {\rm(ii)} the relation $\{(x,x')\in X\times X: {\rm supp}\,x \subseteq {\rm supp}\,x'\}$, denoted by $\leqslant :X\longrightarrow X$, defined in \cite{hosseinabadi}, is an equivariant relation on $X$.

  \medskip

  {\rm(iii)} {\em the support relation} $\{(x,d)\in X\times \mathbb{D}:d\in {\rm supp}\,_{_{X}}x \}$, denoted by \linebreak $supp: X\longrightarrow \mathbb{D}$, is an equivariant element of $X\longrightarrow \mathbb{D}$.
\medskip
 
 {\rm (iv)} the  \emph{freshness relation} $\{(x,y)\in X\times Y: {\rm supp}\,x\cap{\rm supp}\,y=\emptyset\}$, denoted by  $\sharp_{_{X,Y}}:X\longrightarrow Y$, defined in \cite{Pitts2}, is an equivariant relation on $X\times Y$ and it is said that $x$ is \emph{fresh} for $y$. 
     To simplify, we denote the equivariant relation $\sharp_{_{X, X}}: X\longrightarrow X$ by $\sharp_{_{X}}$. 
 \end{example}

 Among the various examples in the preceding example, the freshness relation is a significant and useful one \cite{PH}. We discuss further conditions for freshness relation in  certain circumstances in the next theorem.
 
\begin{theorem}\label{sharp nonempty} Given non-empty nominal sets $X$ and $Y$,   
	
	{\rm (i)} the relation $\sharp_{_{X}}\neq \emptyset$ and ${\rm Dom}\sharp_{_{X}}=X$.

	{\rm (ii)} the relation $\sharp_{_{X}}$ is  symmetric, that is $\sharp_{_{X}}^{-1}=\sharp_{_{X}}$.
	
	{\rm (iii)} if $\mathcal{Z}(X)\neq \emptyset$, then $\sharp_{_{X,Y}}$ is a surjective relation. 
	
	{\rm (iv)} the relation $\sharp_{_{\mathbb{D}, X}}$ is surjective. 
	
	{\rm (v)} for the nominal set ${\rm Perm}(\mathbb{D})$, if $(d,\pi _{1}),(d,\pi_{2})\in \sharp_{_{\mathbb{D}, {\rm Perm}(\mathbb{D})}}$, then $(d, \pi_1\circ\pi_2) \in \sharp_{_{\mathbb{D}, {\rm Perm}(\mathbb{D})}}$. 
	
	{\rm(vi)} the nominal set $X$ is discrete if and only if $\sharp_{_{X}}$ is a reflexive relation. 	\end{theorem}
 \begin{proof}
 {\rm (i)}	If $\mathcal{Z}(X)\neq \emptyset$, then $(\mathcal{Z}(X)\times X)\cup (X\times \mathcal{Z}(X))\subseteq \sharp_{_{X}}$. If $\mathcal{Z}(X)=\emptyset$, then for each $x\in X$ with ${\rm supp}\,x\neq \emptyset$ there exists $\pi \in {\rm Perm}(\mathbb{D})$ such that ${\rm supp}\,\pi x\cap {\rm supp}\,x=\emptyset$ and hence, $(x, \pi x)\in \sharp_{_{X}}$. 
 
  \medskip
 
 {\rm (ii)} It is trivial.
 
  \medskip

 {\rm(iii)}  The zero elements are fresh for every $y\in Y$, so $\sharp_{_{X,Y}}$ is surjective. Furthermore, if $X=\{\theta\}$ is a singleton nominal set, then $\sharp_{_{\lbrace\theta\rbrace, Y}}$ is also injective.

  \medskip

 {\rm(iv)}  Using the Choose-a-Fresh-Name Principle, there exists $d\in\mathbb{D}$ with $d\notin {\rm supp}_{_{X}}\,x$, for every $x\in X$, which means $(d,x)\in \sharp_{_{\mathbb{D}, X}}$, the result is obtained.
 
     \medskip
  
   {\rm(v)} Since $d\notin {\rm supp}\,\pi_1 \cup {\rm supp}\,\pi_2$ and ${\rm supp}\,\pi _{1}\circ \pi_{2}\subseteq{\rm supp}\,\pi _{1}\cup {\rm supp}\,\pi _{2}$, we get the desired result.
   
    \medskip
   
   {\rm(vi)} The relation $\sharp_{_{X}}$ is  reflexive if and only if $(x,x)\in \sharp_{_{X}}$, for every $x\in X$, if and only if ${\rm supp}\,x=\emptyset$, for every $x\in X$.   
 \end{proof}

The following theorem is simple to prove.
 \begin{theorem}
The composition of two binary equivariant relations is an equivariant relation.
\end{theorem}

 \begin{corollary}
 	Nominal sets ($G$-sets) and the equivariant relations between them, together with the relational composition and diagonal relations as identities, form a category ${\bf Rel}({\bf Nom})$ ({\bf Rel}-$(G$-{\bf Set}$)$).
 \end{corollary}
 
  Here, we are going to study some categorical properties in this category.
  
  \begin{theorem}\label{pXY is Gset}
    Let $X$ and $ Y$ be two $G$-sets. Then, the set $\mathcal{R}(X,Y)$ with the action     
   $*: G\times{\mathcal{R}(X,Y)}\longrightarrow{\mathcal{R}(X,Y)}$ defined by $$(g,R)\rightsquigarrow g\ast R=\{(x, g y): (g^{-1}{x},y)\in R\},$$ is a $G$-set.
   \end{theorem}
   
    \begin{proof}
    For every $x\in X$ and $g_{1}, g_{2}\in G$, we have:
  \begin{align*}\label{lcr}
  (x,y)\in{({g_{1}}{g_{2}})\ast{R}}\ 
  &{\Longleftrightarrow}\, \ {{\exists}y'\in Y;\ 
  y=({g_{1}}{g_{2}})y'\text{ and }
  ~(({g_{1}}{g_{2}})^{-1}x,y')\in R} \\
  &{\Longleftrightarrow}\  {{ \exists}y'\in Y;\  (({g_{2}}^{-1} 
  {g_{1}}^{-1})x,y')}\in R\\
  &\Longleftrightarrow{{\exists}y'\in Y;
  ({g_{2}}^{-1}({g_{1}}^{-1}x),y')}\in{R}\\
  &\Longleftrightarrow{{\exists}y'\in Y;\  ({g_{1}}^{-1}x,g_{2}y')
  \in{g_{2}}\ast{R}}\\
  &\Longleftrightarrow {\exists} y'\in Y; \ (x,g_{1}(g_{2}y'))\in{g_{1}}\ast{(g_{2}\ast{R})}\\
  &\Longleftrightarrow{{\exists} y'\in Y; \ ~y=g_{1} (g_{2} {y'}),~(x,y)\in{g_{1}}\ast{(g_{2}\ast{R})}}\\
  &\Longleftrightarrow{(x,y)\in{g_{1}}\ast{(g_{2}\ast{R})}}.
  \end{align*}
  
   So, $({g_{1}}{g_{2}})\ast{R}=g_{1}\ast{(g_{2}\ast{R})}$. Also, $$e*R=\{(x,e\,y): (e\,x,y)\in R\}=\{(x,y): (x,y)\in R\}=R.{\qedhere}$$
  \end{proof}
More characterizations of equivariant relations between two $G$-sets are provided in the subsequent theorem.
  \begin{theorem}\label{rrr}
  	Let $X ,Y$ be two $G$-sets, $R\in \mathcal{R}(X,Y)$ and $g\in G$. Then, the following statements are equivalent.
  	
  	{\rm(i)} The relation $R$ is equivariant;
  	
  	{\rm(ii)} The relation $R^{-1}$ is  equivariant;
  	
  	{\rm(iii)}  For every $y\in {\rm Im}R$, we have $g ({\overleftarrow R}(y))={\overleftarrow R}(g y)$;
  	
  	{\rm(iv)} For every $x\in {\rm Dom}R$, we have $g ({\overrightarrow R}(x))={\overrightarrow R}(g x)$;
  	
  	{\rm(v)} $g*R=R$;
  	
  	{\rm(vi)} The relation $R\in \mathcal{R}(X, Y)$ is equivariant; i.e. $R$ is a zero element of $\mathcal{R}(X, Y)$ for the action $``*"$ defined in {\rm Theorem \ref{pXY is Gset}}.
  \end{theorem}

  \begin{proof}
  	(i$\Rightarrow$ii) Given each $(x,y) \in R^{-1}$ and $g \in G$, we have $(y,x) \in R$. Since
  	R is equivariant $(gy,gx) \in R$ and so $(gx,gy) \in R^{-1}$.
  
  \medskip
  	
  	${\rm(ii\Rightarrow  iii)}$ Suppose $R^{-1}$ is an equivariant relation. Then we have:
  	\begin{align*}
  		x\in g({\overleftarrow R}(y)) & \Longleftrightarrow  g^{-1}x\in {\overleftarrow R}(y)\\
  		& \Longleftrightarrow   (g^{-1}x, y)\in R\\
  		&\Longleftrightarrow   (y, g^{-1}x)\in R^{-1}\\
  		& \Longleftrightarrow  (g y, x)\in R^{-1}\\
  		& \Longleftrightarrow   (x, gy)\in R\\
  		&\Longleftrightarrow  x\in {\overleftarrow R}(gy).
  	\end{align*}
  
   \medskip
  	
  	(iii$\Rightarrow$iv) 
  	\begin{align*}
  		y\in g({\overrightarrow R}(x)) & \Longleftrightarrow  g^{-1}y\in {\overrightarrow R}(x)\\
  		&  \Longleftrightarrow   (x, g^{-1} y)\in R\\
  		&  \Longleftrightarrow   x\in \overleftarrow{R}(g^{-1}y)\\
  		& \Longleftrightarrow  x\in g^{-1} \overleftarrow{R}(y)\\
  		& \Longleftrightarrow   g x\in \overleftarrow{R}(y)\\
  		& \Longleftrightarrow  y\in \overrightarrow{R}(gx).
  	\end{align*}

  	 \medskip

  	$\rm(iv\Rightarrow \rm v)$ Suppose $g ({\overrightarrow R}(x))={\overrightarrow R}(g x)$, for every $x\in{\rm Dom}R$ and $g\in G$. Then  we have:
  	
  	\begin{align*}
  		(x,y)\in g*R \, &  \Longleftrightarrow \exists \, y'\in{Y}; \, ~y=gy' \ \text{ and }~(g^{-1}x,y')\in R\\
  		&  \Longleftrightarrow \exists \,  y'\in{Y};  ~ y'\in {\overrightarrow R}(g^{-1}x)\\
  		&   \Longleftrightarrow \exists \,  y'\in{Y};  ~ y'\in g^{-1}({\overrightarrow R}(x))\\
  		&   \Longleftrightarrow \exists \,  y'\in{Y}; ~ y=gy' \text{and} \ \, gy'\in {\overrightarrow R}(x)\\
  		&  \Longleftrightarrow   y\in {\overrightarrow R}(x)\\
  		&  \Longleftrightarrow \, (x,y)\in R.
  	\end{align*}

  	 \medskip
  	
  	$\rm(v \Rightarrow \rm vi)$ It is trivial.
  
   \medskip
  	
  	$\rm(vi \Rightarrow \rm i)$ Suppose $g*R=R$. Then we have:
  	\begin{align}
  		(x,y)\in R&\Longleftrightarrow\ (x,y)\in g*R;~ \text{for every}~g\in G\tag*{}\\
  		&\Longleftrightarrow\ \exists y'\in{Y};~y=gy'\text{ and }~ (g^{-1}x,y^{\prime})\in R\tag*{}\\
  		&\Longleftrightarrow\ \exists y'\in{Y};~ y'=g^{-1}y\text{ and }~y'\in {\overrightarrow R}(g^{-1}x)\tag*{}\\
  		&\Longleftrightarrow\ g^{-1}y\in {\overrightarrow R}(g^{-1}x)\tag*{}\\
  		&\Longleftrightarrow\ (g^{-1}x,g^{-1}y)\in R; ~\text{for every}~g\in G. \tag*{\qedhere}
 	\end{align}
  \end{proof}

   \begin{corollary}
{\rm (i)} According to {\rm Theorem \ref{rrr}(v)}, for any  nominal sets (or in general  $G$-sets) $X$ and $Y$, the zero elements of  $\mathcal{R}(X,Y)$ (or equivalently, the relations with empty support) are exactly the equivariant relations from $X$ to $Y$.

 \medskip

{\rm(ii)} For any  nominal sets $X$ and $Y$, $\mathcal{Z}(\mathcal{R}_{{\rm fs}}(X, Y))=\mathcal{Z}(\mathcal{R}(X, Y))$.

\end{corollary}

  \begin{theorem}\label{in}
 Let $X$ and $Y$ be $G$-sets and $R\in \mathcal{R}(X,{Y})$. Then, $g{\ast{R}}=g\cdot R$, for every $g\in {G}$.
 \end{theorem}
 
\begin{proof}
Suppose $X$ and $Y$ are $G$-sets and $R\in \mathcal{R}(X,{Y})$, for every $x\in X$,  we have:
\begin{align}
y\in (\overrightarrow{g * R})(x) &\Longleftrightarrow (x,y)\in g * R\tag*{}\\
&\Longleftrightarrow \exists y'\in Y;~y=g y',~ (g^{-1}x,y' )\in R \tag*{}\\
&\Longleftrightarrow \exists y'\in Y ;~y'\in \overrightarrow{R}(g^{-1}x)\tag*{}\\
&\Longleftrightarrow \exists y'\in Y;~y'=g^{-1}  y,~g^{-1}  y\in \overrightarrow{R}(g^{-1}x)\tag*{}\\
&\Longleftrightarrow y\in{g\overrightarrow{R}(g^{-1}x)}\tag*{}\\
&\Longleftrightarrow y\in (\overrightarrow{g \cdot R})(x).
\tag*{\qedhere}
\end{align}
\end{proof}

  \begin{proposition}\label{comper}
 Suppose $X,Y$ are nominal sets and $R\in \mathcal{R}_{{\rm fs}}(X, Y)$. So
 
 {\rm (i)} if $S\in \mathcal{P}_{\rm fs}(X)$, then ${\rm supp}\,\overrightarrow{R}(S)\subseteq {\rm supp}\,R \cup {\rm supp}\,S$. 
 
 {\rm (ii)} if $S'\in \mathcal{P}_{\rm fs}(Y)$, then ${\rm supp}\,\overleftarrow{R}(S')\subseteq {\rm supp}\,R \cup {\rm supp}\,S'$.
\end{proposition}
\begin{proof}
 {\rm (i)} Let $\pi \in {\rm Perm}(\mathbb{D})$ with $\pi (d) =d$, for all $d\in{\rm supp}\,R \cup {\rm supp}\,S$. 
 Thus, $\pi \ast R=R$ and $\pi S=S$. Since $\pi \ast R=R$,  by the definition of  $``*"$ in Theorem \ref{pXY is Gset},  we get that $(\pi x, y)\in \pi \ast R=R$ if and only if $(x,  \pi^{-1} y)\in R$. Now, we have:

 \begin{align*}
  y\in \pi ({\overrightarrow R}(S)) & \Longleftrightarrow  \pi^{-1}y\in {\overrightarrow R}(S)\\
   &\Longleftrightarrow   \exists \, x\in S;~\, (x, \pi^{-1} y)\in R\\
   &\Longleftrightarrow  \exists \, x\in S;~\,  (\pi x, y) \in R\\
  & \Longleftrightarrow  (x', y)\in R, \ x'=\pi x\in S \\
   & \Longleftrightarrow   y \in \overrightarrow{R}(S).
  \end{align*}
 
  {\rm (ii)} The proof is similar to part (i).
\end{proof}
 
\begin{corollary}\label{qaz}
Suppose $X,Y$ are nominal sets. If $R\in \mathcal{R}_{{\rm fs}}(X, Y)$ and $(x,y)\in R$, then

{\rm(i)} ${\rm supp}\,\overrightarrow{R}(x)\subseteq {\rm supp}\,R \cup {\rm supp}\,x$.

{\rm(ii)} ${\rm supp}\,\overleftarrow{R}(y)\subseteq {\rm supp}\,R \cup {\rm supp}\,y$. 
\end{corollary}
\begin{proof}
Let $S=\{x\}$ and $S'=\{y\}$. Then,  applying Proposition \ref{comper}, we get the result.
\end{proof}
  
 \begin{corollary}\label{n}
Suppose $X$ and $ Y$ are nominal sets and $R\in \mathcal{R}(X, Y)$ is equivariant.

{\rm(i)} If $X'\in \mathcal{Z}(\mathcal{P}(X))$ and $Y'\in \mathcal{Z}(\mathcal{P}(Y))$, then
 $\overrightarrow{R}(X')$ and $\overleftarrow{R}(Y')$ are, respectively, equivariant subsets of $Y$ and $X$.

{\rm(ii)}  ${\rm supp}\,\overrightarrow{R}(x)\subseteq{\rm supp}\,x $. 

{\rm(iii)}  ${\rm supp}\,\overleftarrow{R}(y)\subseteq {\rm supp}\,y$. 
 \end{corollary}
 
\begin{proof}
{\rm(i)} Since $X'\in \mathcal{Z}(\mathcal{P}(X))$ and $R\in \mathcal{R}(X, Y)$ is an equivariant relation, we have ${\rm supp}\,X'={\rm supp}\,R=\emptyset$. Now, applying Proposition \ref{comper}(i), we get that ${\rm supp}\,\overrightarrow{R}(X')=\emptyset$. Analogously, one can prove that $\overleftarrow{R}(Y')$ is an equivariant subset of $Y$. 

 \medskip

{\rm(ii)} Since $R$ is equivariant, we have ${\rm supp}\,R=\emptyset$. Now, applying Corollary \ref{qaz}(i), ${\rm supp}\,\overrightarrow{R}(x)\subseteq{\rm supp}\,x $. 

 \medskip

{\rm (iii)} The proof is similar to (ii).
\end{proof}

 \begin{lemma}\label{rrrr}
Given two nominal sets $X,Y$ and injective relation $R\in \mathcal{R}_{\rm fs}(X, Y)$ and $(x,y)\in R$, we can deduce the following:

{\rm (i)} ${\rm supp}\,x\subseteq {\rm supp}\,\overrightarrow{R}(x)\cup {\rm supp}\,R $.

{\rm (ii)} ${\rm supp}\,x\subseteq {\rm supp}\,y\cup {\rm supp}\,R$.

{\rm(iii)} ${\rm supp}\,x= {\rm supp}\,\overleftarrow{R}(y)$.
\end{lemma}

\begin{proof}
{\rm (i)} Let $d,d'\notin  {\rm supp}\,\overrightarrow{R}(x)\cup {\rm supp}\,R$. Then $(d\ d')\overrightarrow{R}(x)=\overrightarrow{R}(x)$ and $R=(d\ d')\cdot R$. By Remark \ref{popol}(i), we have $\overrightarrow{R}(x)=\overrightarrow{((d\ d')\cdot R)}(x)=(d\ d')(\overrightarrow{R} (d\ d')x)$. Thus, $(d\ d') \overrightarrow{R}(x)=\overrightarrow{R} (d\ d')x$ and so $\overrightarrow{R}(x)=\overrightarrow{R} (d\ d')x$. Now, since $R$ is injective, we get that $(d\ d')x=x$. So, by Remark \ref{supp for G} (i), ${\rm supp}\,\overrightarrow{R}(x)\cup {\rm supp}\,R$ supports $x$.

 \medskip

{\rm (ii)} Let $d,d'\notin  {\rm supp}\,y \cup {\rm supp}\,R$. Then, $(d\ d') y= y$ and $R=(d\ d')\cdot R$. By Remark  \ref{popol}(i), we have $(d\ d')\overrightarrow{R}(x)=\overrightarrow{R} (d\ d')x$. Thus, $y\in \overrightarrow{R}(x)$ implies that $y=(d\ d')y\in (d\ d')\overrightarrow{R}(x)=\overrightarrow{R} (d\ d')x$. So, $(x,y), ((d\ d')x, y)\in R$.
Now, since $R$ is injective, we get that $(d\ d')x=x$. So, by Remark \ref{supp for G}(i), ${\rm supp}\,y\cup {\rm supp}\,R$ supports $x$.

 \medskip

{\rm(iii)} Since $R$ is injective and $x\in \overleftarrow{R}(y)$, we get that $\overleftarrow{R}(y)=\{x\}$. Thus, ${\rm supp}\,x={\rm supp}\,\overleftarrow{R}(y)$. 
\end{proof}

\begin{corollary}\label{yhm}
Let $X,Y$ be two nominal sets,  $R\in \mathcal{R}(X, Y)$ be an equivariant injective relation, and $(x,y)\in R$. Then,
 
 {\rm (i)} ${\rm supp}\,x= {\rm supp}\,\overrightarrow{R}(x)$.
 
 {\rm (ii)}  ${\rm supp}\,x\subseteq {\rm supp}\,y$.
\end{corollary}

\begin{proof}
Notice that, since $R$ is an equivariant relation, ${\rm supp}\,R=\emptyset$. So,

{\rm(i)} applying Lemma \ref{rrrr}(i), we have $ {\rm supp}\,x\subseteq {\rm supp}\,\overrightarrow{R}(x)$ and Corollary \ref{n}(ii) implies that $ {\rm supp}\,\overrightarrow{R}(x)\subseteq  {\rm supp}\,x$. Thus ${\rm supp}\,\overrightarrow{R}(x)={\rm supp}\,x$.

 \medskip

{\rm(ii)} the result follows by Lemma \ref{rrrr}(ii).
\end{proof}

\begin{corollary}
Let $R\in \mathcal{R}(X, \mathbb{D})$ 
be an equivariant injective relation and $(x, d)\in R$. Then, ${\rm supp}\,x\subseteq \{d\}$ and so $x\in \mathcal{Z}(X)$ or ${\rm supp}\,x=\{d\}$.
\end{corollary}

 The following example shows that the converse of Corollary \ref{yhm}(ii) is not true in general.

\begin{example}
Suppose $R\in \mathcal{R}(\mathbb{D}, \mathbb{D}^{(2)})$ is a relation defined by $$R=\{(d,(x,y)): d=x\, \vee \, d=y,\ x\neq y\}.$$ For all $(d, (x,y))\in R$, we have $\{d\}={\rm supp}\,d\subseteq \{x,y\}={\rm supp}\,(x,y)$, but $R$ is not injective, since $(d,(d,d')), (d',(d,d'))\in R$.  
\end{example}

\begin{proposition}
Let $X$ be a nominal set. Then,

{\rm(i)} the relation $\leqslant$, given in Example \ref{suprel}(ii), is a reflexive and transitive relation on $X$.

{\rm(ii)} if $R$ is a non-empty injective equivariant relation on $X$, then $R\subseteq \leqslant$.

{\rm(iii)} the set $\mathcal{R}^{e}_{_{inj}}$ of all injective equivariant relations on $X$ is a nominal set and $\leqslant$ is an upper bound for $\mathcal{R}^{e}_{_{inj}}$.
\end{proposition}

\begin{proof}
{\rm(i)} That is clear.

 \medskip
 
{\rm(ii)} Let $(a,b)\in R$. Since $R$ is injective, applying Corollary \ref{yhm}(ii), ${\rm supp}\,a\subseteq {\rm supp}\,b$. So, $R\subseteq \leqslant$.

 \medskip
 
{\rm(iii)} Let $\pi \in {\rm Perm}(\mathbb{D})$ and $R\in \mathcal{R}^{e}_{_{inj}}$ with $(x, y), (x', y)\in \pi* R$. Since $R$ is equivariant, $(\pi^{-1} x, \pi^{-1} y), (\pi^{-1} x', \pi^{-1} y)\in R$. Now, since $R$ is injective, $\pi^{-1} x=\pi^{-1} x'$ and so $x=x'$.  Also, $\leqslant$ is an upper bound of $\mathcal{R}^{e}_{_{inj}}$ by (ii).
\end{proof}

\begin{corollary}
{\rm(i)} The only equivariant injective relation on $\mathbb{D}$ is $\Delta _{\mathbb{D}}$. To examine, let $R$ be an injective equivariant relation on $\mathbb{D}$ and $(d, d')\in R$. Then, ${\rm supp}\,d\subseteq {\rm supp}\,d'$ and so $d=d'$.

{\rm(ii)} Let the relation $R \in \mathcal{R}(\mathbb{D}, \mathbb{D}^{(2)})$  be equivariant and injective, and $(x,y)\in R$. Then,  there exist $a, b\in \mathbb{D}$ where $a\neq b$ such that either $x=a$ or $x=b$ and  $y=(a,b)$, but both $(a, (a,b))$ and $(b, (a,b))$ can not belong to $R$.

\end{corollary}

\begin{proposition}\label{faroh}
Suppose $X$ is a nominal set and $ R\in \mathcal{R}(X, X)$ is an equivariant injective relation and $(x,y)\in R$. If $R$ is a symmetric relation,  then

{\rm (i)}  ${\rm supp}\,x={\rm supp}\,y$.

{\rm (ii)} ${\rm supp}\,\overrightarrow{R}(x)={\rm supp}\,y$.
\end{proposition}

\begin{proof}
{\rm (i)} Suppose $(x,y)\in R$.  Since $R$ is symmetric, $(y,x)\in R$. By Corollary \ref{yhm}(ii) and the assumption of injectivity of $R$, we get that ${\rm supp}\,x={\rm supp}\,y$.

 \medskip

{\rm (ii)} The proof follows from part (i) and Corollary \ref{yhm}(i).
\end{proof}

According to \cite[Lemma 2.12 (iii)]{Pitts2}, if $f:X\rightarrow Y$ is a surjective equivariant map, in which $X$ is a nominal set and $Y$ is a ${\rm Perm}(\mathbb{D})$-set, then $Y$ is a nominal set. The example that follows demonstrates that this is untrue when $f$ is an equivariant relation.

\sloppypar\noindent\begin{example}
Let $R: \mathbb{D}\longrightarrow [\mathcal{P}(\mathbb{D})\setminus \emptyset]$ be a relation defined by $R=\{(d, A): d\in A\}$. Then, $R$ is surjective and equivariant. Notice that, by Remark \ref{P(X)-is-act}(iii), $\mathcal{P}(\mathbb{D})$ is not a nominal set.
\end{example}

\begin{proposition}\label{R_* and R^*}
Let $X$ and $Y$ be two nominal sets and $R\in \mathcal{R}(X, Y)$ be equivariant. Then, 

{\rm(i)} the map $\overrightarrow{R}:\mathcal{P}_{{\rm fs}}(X)\rightarrow \mathcal{P}_{{\rm fs}}(Y)$ mapping each $A\in\mathcal{P}_{{\rm fs}}(X)$ to $\overrightarrow {R}(A)$ is  equivariant.

{\rm(ii)} the map $\overleftarrow{R}:\mathcal{P}_{{\rm fs}}(Y)\rightarrow \mathcal{P}_{{\rm fs}}(X)$ defined by $\overleftarrow{R}(B)=\{x\in{X}: \exists y\in {B; (x,y)\in{R}}\}$, for each $B\in \mathcal{P}_{{\rm fs}}(Y)$, is  equivariant.
\end{proposition}

\begin{proof}
{\rm(i)} Let $\pi \in {\rm Perm}(\mathbb{D})$. Then, we show that $\pi  \overrightarrow{R}(A)=\overrightarrow{R}(\pi A)$.
To prove, let $\pi y\in \pi \overrightarrow{R} (A)$. Then, there exists $x'\in A$ with $(x',y)\in R$. Since $R$ is equivariant, $(\pi x', \pi y)\in R$. Thus, $\pi x'\in \pi A$, and so, $\pi y\in  \overrightarrow{R}(\pi A)$. Similarly, one can see that $\overrightarrow{R}(\pi A)\subseteq \pi \overrightarrow{R}(A)$.

 \medskip

{\rm(ii)} The proof is similar to part (i).
\end{proof}

\begin{lemma}\label{indecom is surj}
Let $X$ and $Y$ be two $G$-sets and $R \in \mathcal{ R}(X, Y )$ be equivariant. Then,

{\rm(i)} if $R$ is a partial surjective map and $X$ indecomposable, then $\overrightarrow{R}(X)$ is indecomposable.

{\rm(ii)} if $R$ is non-empty and $Y$ is indecomposable, then $R$ is surjective.
\end{lemma}
\begin{proof}
{\rm(i)} On the contrary, suppose there exist disjoint equivariant subsets $A$ and $B$ of $Y$ with
$\overrightarrow{R}(X)=A \cup B$. Since $R$ is a partial surjective map, $\overleftarrow{R}(A)$ and $\overleftarrow{R}(B)$ are non-empty equivariant subsets of $X$. Since $X$ is indecomposable, by Lemma \ref{cyclic nom}, $X=\overleftarrow{R}(A)=\overleftarrow{R}(B)$ which is a contradiction. This is because, 
$\overleftarrow{R}(A)\cap \overleftarrow{R}(B)=\overleftarrow{R}(A\cap B)=\emptyset$, by Remark \ref{prop of BB}. 

 \medskip

{\rm(ii)} By Lemma \ref{cyclic nom}, there exists $y\in Y$ with $Y=Gy$. Let $t\in Y$ and $(x, z)\in R$. Then, there exist $g_1, g_2\in G$ with $t=g_1 y$ and $z=g_2 y$. Now, since $R$ is equivariant, we get that $(g_1 g_2^{-1} x, t)=(g_1 g_2^{-1} x, g_1 y)=(g_1 g_2^{-1} x, g_1 g_2^{-1} z)\in R$; 
 meaning that $R$ is surjective.
\end{proof}

\begin{proposition} \label{opp}
Let $X,Y$ be nominal sets and $R\in \mathcal{R}(X, Y)$ be equivariant. If $R$ is epic, then $R$ is surjective.
\end{proposition}
\begin{proof}

 Notice that, if $Y$ is indecomposable, then $R$ is surjective, by Lemma \ref{indecom is surj}(ii). So, suppose $Y$ is decomposable and $R$ is epic. If $R$ is not surjective, then there exists  $y\in Y$, but $y\notin {\rm Im}R$. Take $Z=\{\theta_1,\theta_2\}$ be a discrete nominal set and
 $R_1:{\rm Im}R\longrightarrow \{\theta_1\}$ define $R_1={\rm Im}R \times \{\theta_1\}$ and $R_2:{\rm Im}R\longrightarrow \{\theta_1,\theta_2\}$ define $R_2=({\rm Im}R\times \{\theta_1\})\cup ({\rm Perm}(\mathbb{D})y\times \{\theta_2\})$. 
 Notice that, since ${\rm Im}R$ is an equivariant subset of $Y$, we get that $R_1$ and $R_2$ are equivariant relations. For given $x\in {\rm Dom}R$, we have $\overrightarrow{R}(x)\subseteq {\rm ImR}$ and so $R_1\circ R=R_2\circ R$. Since $R$ is epic, we get $R_1=R_2$, which is a contradiction.
 \end{proof}

 Now, we show that the converse of  Proposition \ref{opp} is not correct, in general.

\begin{example}\label{surj not epic}
 Suppose $\mathbb{D}$ is a nominal set. Consider the equivariant relations $\sharp_{\mathbb{D}}=\{(d,d'): d\neq d'\}=\mathbb{D}^{(2)}$ and $R=\mathbb{D}^2=\mathbb{D}\times \mathbb{D}$. By Theorem \ref{sharp nonempty}(iv), $\sharp_{\mathbb{D}}$ is surjective. We have $R\circ \sharp_{\mathbb{D}}=R=\sharp_{\mathbb{D}} \circ \sharp_{\mathbb{D}}$, but $R\neq \sharp_{\mathbb{D}}$  meaning that $\sharp_{\mathbb{D}}$ is not epic. 
\end{example}

\begin{note}
{\rm(i)} If a relation $R$ has a right inverse, then $R$ is surjective.
This is because morphisms having a right inverse are epimorphisms and so
 by Proposition \ref{opp}, we get the result.

 \medskip

{\rm(ii)} Example \ref{surj not epic} also shows that the converse of (ii) does not hold. Indeed, if $S$ is an equivariant relation on $\mathbb{D}$ with $\sharp_{\mathbb{D}} \circ S=id_{{\mathbb{D}}}$,  since $(d,d)\in \sharp_{\mathbb{D}} \circ S$ and $\sharp_{\mathbb{D}}$ is surjective, there exists $d' \neq d$ with $(d',d)\in \sharp_{\mathbb{D}}$ and $(d,d')\in S$. For given $d''\neq d', d$, we have $(d', d'')\in \sharp_{\mathbb{D}}$. Thus $(d, d'')\in \sharp_{\mathbb{D}} \circ S=id_{\mathbb{D}}$ meaning that $d=d''$ and this is a contradiction.
\end{note}

By the same scheme of \cite{john} but different in details we have the  following theorem.

\begin{theorem} \label{er}
Let $X,Y$ be nominal sets and $R\in \mathcal{R}(X, Y)$ be equivariant. Then, the following statements are equivalent.

{\rm(i)} The relation $R$ is monic and ${\rm Dom}R=X$.

{\rm(ii)} The map $\overrightarrow{R}:\mathcal{P}_{_{\rm fs}}(X) \rightarrow \mathcal{P}_{_{\rm fs}}(Y)$ is  injective  and ${\rm Dom}R=X$.

{\rm(iii)} The relation $R$ is total injective.
\end{theorem}

\begin{proof}
(i$\Rightarrow$ii) Suppose $\overrightarrow{R}(U)=\overrightarrow{R}(V)$ for some $U,V\in \mathcal{P}_{_{\rm fs}}(X)$. We show $U=V$. To prove that, we consider finitely supported relations $S=\{(*,u):u\in U\}$ and $T=\{(*,v):v\in V\}$. Then, clearly, $R\circ S=R\circ T$. Hence  $S=T$, since $R$ is monic. Thus $U=V$. 

 \medskip

(ii$\Rightarrow$iii) Notice that, $X-\{x\}\in \mathcal{P}_{_{\rm fs}}(X)$, for every $x\in X$. To do so, let $C$ be a finite support of $x$ and $d_1, d_2\notin C$. Then, $(d_1 \ d_2)x=x$. Now, for all $y\in X-\{x\}$, we have 
$x=(d_1\ d_2)x\neq (d_1\ d_2)y\in X-\{x\}$ meaning that $C$ is a finite support for $X-\{x\}$.
 Since $\overrightarrow{R}$ is injective and $X\neq X-\{x\}$, we have $\overrightarrow{R}(X)\neq \overrightarrow{R}(X-\{x\})$. For given $x\in X$, since  ${\rm Dom}R=X$, there exists $y\in Y$ with $(x,y)\in R$. Now, if there exists $x''\in X-\{x\}$ with $(x'', y)\in R$, then $y\in \overrightarrow{R}(X-\{x\})$ which is a contradiction.
 
  \medskip

(iii$\Rightarrow$i) First, notice that since $R$ is total, we get that ${\rm Dom}R=X$. Now, suppose $R_1, R_2\in \mathcal{R}(Z, X)$ are equivariant with $R\circ R_1=R\circ R_2$. Let $(z,x)\in R_1$. Then, by the assumption, there exists $y\in Y$ with $(x,y)\in R$. Thus $(z,y)\in R\circ R_1=R\circ R_2$ and so there exists $x'\in X$ with $(z,x')\in R_2$ and $(x',y)\in R$. Since $R$ is injective and $(x,y), (x',y)\in R$, we get that $x=x'$. So, $(z,x)=(z,x')\in R_2$ implies that $R_1\subseteq R_2$. Analogously, $R_2\subseteq R_1$, and we are done.
\end{proof}

\begin{lemma}\label{partial sur is one one}
Let $X,Y$ be nominal sets and $R\in \mathcal{R}(X, Y)$ be equivariant. Then, the following statements are equivalent.

{\rm(i)} The map $\overleftarrow{R}:\mathcal{P}_{_{\rm fs}}(Y)\rightarrow \mathcal{P}_{_{\rm fs}}(X)$ is injective and ${\rm Im}R=Y$.

{\rm(ii)} The map $R$ is  partial surjective.
\end{lemma}
\begin{proof}
The proof is similar to Theorem \ref{er}. 
\end{proof}

\begin{remark}
If $R$ is a partial surjective map, then similar to the proof of (iii$\Rightarrow$i) in Theorem \ref{er}
 $R$ is epic, but the converse is not true (see Example \ref{partial et epic}(ii)).
\end{remark}

\begin{corollary}\label{parsur}
Suppose $R\in \mathcal{R}(X, Y)$ is a partial surjective equivariant map between nominal sets $X$ and $Y$. Let $(x,y)\in R$. Then,

{\rm(i)} ${\rm supp}\,\overleftarrow{R}(y)={\rm supp}\,y$.

{\rm(ii)} ${\rm supp}\,y\subseteq {\rm supp}\,x$.
\end{corollary}
\begin{proof}
{\rm(i)} By Lemma \ref{partial sur is one one}, $\overleftarrow{R}$ is an injective equivariant map. By Corollary \ref{yhm}(i), we
get the result.

 \medskip

{\rm(ii)} Similar to the proof of Corollary \ref{yhm}(ii).
\end{proof}

\begin{example}\label{partial et epic}
{\rm(i)} The relation $\sharp_{\mathbb{D}}=\{(d,d'): d\neq d'\}$ is a surjective equivariant relation on $\mathbb{D}$, but the map $\overleftarrow{\sharp_{\mathbb{D}}}:\mathcal{P}_{_{\rm fs}}(\mathbb{D})\rightarrow \mathcal{P}_{_{\rm fs}}(\mathbb{D})$ is not injective. To do so, let $A, B\in \mathcal{P}_{_{\rm fs}}(\mathbb{D})$. Then $A$ and $B$ are finite or cofinite, by Remark \ref{P(X)-is-act}(iii). Take $A=\{d_1,d_2\}$ and $B=\{d_3,d_4\}$. Now, $\overleftarrow{\sharp_{\mathbb{D}}}(A)=
\{d\in \mathbb{D}: (\exists a\in A)\,   a \sharp_{\mathbb{D}} d\}=\mathbb{D}$ and similarly $\overleftarrow{\sharp_{\mathbb{D}}}(B)=\mathbb{D}$. So $\overleftarrow{\sharp_{\mathbb{D}}}(A)=\overleftarrow{\sharp_{\mathbb{D}}}(B)$, but $A\neq B$.

 \medskip

{\rm(ii)} Let $R\in \mathcal{R}(\mathbb{D}, \mathbb{D}^{(2)})$ be equivariant defined by $R=\{(d, (d,d')): d\neq d'\}$. Then, $R$ is surjective but  not a partial map. This is because, $(d, (d, d')), (d, (d, d''))\in R$ for $d''\neq d'$. On the other hand, $S=\{((d,d'),d): d\neq d'\}$ is an equivariant relation and $R\circ S=id_{\mathbb{D}^{(2)}}$, meaning that $R$ is epic. So, epic and partial surjective maps are not equivalent.
\end{example}

% % % % % % % % % % % % % % % % % % % % % % % % % % % % % % % % %

\section{Sheaf representation of nominal sets in ${\rm\bf Rel}$({\bf Nom})}

The sheaf representation of nominal sets provides a more general and abstract setting, which can be useful in various areas of mathematics and computer science, such as the study of programming languages with binding constructs.   
Recall from \cite[Theorem 6.8]{Pitts2} that the category  {\bf Nom} can be considered as a sheaf-subcategory of ${\bf Set}^{\mathcal{P}_{\rm f}(\mathbb{D})}$, by the adjunction 
%$$
%\xymatrix{
%{{\bf Set}^{\mathcal{P}_{\rm f}(\mathbb{D})}}\ar@/^1pc/[rr]^{I^*}&\ar@{}[d]!U|\perp
%& {\bf Nom}\ar@/^1pc/[ll]^{I_*}}
%$$
$I_*: {\bf Nom}\longrightarrow {\bf Set}^{\mathcal{P}_{\rm f}(\mathbb{D})}$ and $I^*:  {\bf Set}^{\mathcal{P}_{\rm f}(\mathbb{D})}\longrightarrow {\bf Nom}$ in which $I^* \dashv I_* $ and $\mathbb{D}$ is the set of atomic names. Hence, $\bf Nom$ is a topos.  In this section, although the category ${\rm\bf Rel}$({\bf Nom}) is not a topos, see Remark \ref{topos}, we are going to examine the counterpart of functors $I^*$ and $I_*$, denoted by $\mathcal{P}_{\rm fs}^ *$ and $\mathcal{P}_{\rm fs *}$ respectively,  for their advantages. 

\begin{lemma}
\sloppypar \noindent {\rm (i)} There is an obvious inclusion $\rm ($or forgetful$\rm )$ functor $I :~{\bf Nom}\hookrightarrow~ {\bf Rel}({\bf Nom})$ that is identity on objects and takes an equivariant map $f : X\to Y$ to its underling relation.

\sloppypar \noindent {\rm (ii)} The inclusion  functor $I$  is a left adjoint for the functor ${\mathcal{P}}_{\rm fs}:~{\bf Rel}({\bf Nom})~\longrightarrow~{\bf Nom} $ that takes every object $X\in {\bf Rel}({\bf Nom})$ to ${\mathcal{P}}_{\rm fs}(X)$ and every equivqrint relation $R:X\to Y$ maps to $\overrightarrow{R}:{{\mathcal{P}}_{\rm fs}(X)\to \mathcal{P}}_{\rm fs}(Y)$. That is, we have:
\[
\xymatrix@R=0.5pc {
\mathcal{P}_{\rm fs }:{\bf Rel}({\bf Nom})\ar[r]& {\bf Nom}\\
X\ar@{~>}[r]\ar[dd]_R&\mathcal{P}_{\rm fs }(X)\ar[dd]^{\overrightarrow{R}}\\
\qquad \ar@{~>}[r]&\quad\\
Y\ar@{~>}[r]&\mathcal{P}_{\rm fs }(Y)
}
\]
\end{lemma}

\begin{proof}
{\rm (i)} It is clear.

\medskip
{\rm (ii)} First, it should be noted that $\mathcal{P}_{\rm fs }(X)$, for any nominal set $X$, is a nominal set according to Remark \ref{P(X)-is-act}(i) and Lemma \ref{Xfs is nom}, and that $\overrightarrow{R}$  is an equivariant map according to Proposition \ref{R_* and R^*}(i). As a result, it is simple to verify that $\mathcal{P}_{\rm fs }$ is a functor.
Now, to prove $I\dashv \mathcal{P}_{\rm fs }$, we show that $\eta _{X}:X\to {\mathcal{P}}_{\rm fs}(I (X))$ defined by $\eta _{X}(x)=\{x\}$ is a universal ${\mathcal{P}}_{\rm fs}$-arrow and $\eta=(\eta _{X})_{X\in {\bf Nom}}$ is a natural transformation. Indeed, for every equivariant map $f:X\to {\mathcal{P}}_{\rm fs}(Y)$, in which $Y\in {\bf Rel}({\bf Nom})$, we define the relation $R_f=\{(x,y)\ :\ y\in f(x)\}\in {\mathcal{R}(I(X),Y)}$. Since $f$ is equivariant, so is $R_f$. We also have $\overrightarrow{R_f} \circ \eta_{X}(x)=\overrightarrow{R_f}(\{x\})=f(x)$. That is the following triangle is commutative.
$$ 
\xymatrix{
X\ar[r]^-{\eta_{X}}\ar[rd]_-{ f}& {\mathcal{P}}_{\rm fs}(I(X))\ar[d]^-{\overrightarrow{R_f}={\mathcal{P}}_{\rm fs}(R_f)} &&I(X)\ar@{-->}[d]^-{\exists! R_f}\\
&{\mathcal{P}}_{\rm fs}(Y) &&Y}
$$ 
The uniqueness of $R_f$ with
$\overrightarrow{R_f} \circ \eta_{X}=f$ follows from its definition and the naturality of $\eta$ can be easily checked. 
\end{proof}

\begin{remark}
By Theorem 6.8 of \cite{Pitts2}, the composition  functor $\mathcal{P}_{\rm fs *}:~{\bf Rel}({\bf Nom}){\overset{\mathcal{P}_{\rm fs }}\longrightarrow} {\bf Nom}{\overset{I_{*}}\longrightarrow} {\bf Set}^{{ \mathcal{P}_{\rm f}(\mathbb{D})}}$, defined by 

\[
\xymatrix@R=0.5pc {
\mathcal{P}_{\rm fs *}:{\bf Rel}({\bf Nom})\ar[r]& {\bf Set}^{ \mathcal{P}_{\rm f}(\mathbb{D})}\\
X\ar@{~>}[r]\ar[dd]_R&\mathcal{P}_{\rm fs *}X\ar[dd]^{R_*}\\
\qquad \ar@{~>}[r]&\quad\\
Y\ar@{~>}[r]&\mathcal{P}_{\rm fs *}Y
}
\]
in which $\mathcal{P}_{\rm fs *}X:\mathcal{P}_{\rm f}(\mathbb{D})\to {\bf Set}$ mapping each $A\in \mathcal{P}_{\rm f}(\mathbb{D})$ to the set $\{k\in~ \mathcal{P}_{\rm fs}X\ :~\ {\rm supp}k\subseteq A\}$ and each equivariant relation $R:X\to Y$  to the  the natural transformation $R_*=\{{R_*}_{A}\}_{A\in \mathcal{P}_{\rm f}(\mathbb{D})}$
  defined by $${R_*}_A=\overrightarrow{R}:\mathcal{P}_{\rm fs *}X(A)\to\mathcal{P}_{\rm fs *}Y(A), \text{ for every } {A\in \mathcal{P}_{\rm f}(\mathbb{D})},$$
  assigns every object in {\bf Rel}({\bf Nom}) to a sheaf.
\end{remark}

\begin{theorem}\label{nbvcrei}
	The composion  functor $\mathcal{P}_{\rm fs *}:\bf Rel(\bf Nom)\to{\bf Set}^{ \mathcal{P}_{\rm f}(\mathbb{D})} $ has a left adjoint, which is denoted by ${\mathcal{P}_{\rm fs}}^*$.
\end{theorem}
\begin{proof}
	First we note that since ${ \mathcal{P}_{\rm f}(\mathbb{D})}$ is an up-directed set, the image of every functor $F$ is up-directed, for every presheaf $F\in {\bf Set}^{ \mathcal{P}_{\rm f}(\mathbb{D})}$. Now we consider the assignment
	\[
	\xymatrix@R=0.5pc {
		{\mathcal{P}_{\rm fs}}^*:{\bf Set}^{ \mathcal{P}_{\rm f}(\mathbb{D})}\ar[r]& \bf Rel(\bf Nom)\\
		F\ar@{~>}[r]\ar[dd]_\tau&\overrightarrow{{\rm lim}}_{A\in { \mathcal{P}_{\rm f}(\mathbb{D})}}FA\ar[dd]^{\tau^*} \\
		\qquad \ar@{~>}[r]&\quad\\
		G\ar@{~>}[r]&\overrightarrow{{\rm lim}}_{A\in { \mathcal{P}_{\rm f}(\mathbb{D})}}GA
	}
	\]
	
	\sloppypar in which $\overrightarrow{{\rm lim}}_{A\in { \mathcal{P}_{\rm f}(\mathbb{D})}}FA$ is direct limit (or directed colimit) of the diagram $\{FA\}_{A\in { \mathcal{P}_{\rm f}(\mathbb{D})}}$ which  is the quotient $\bigcup_{A\in { \mathcal{P}_{\rm f}(\mathbb{D})}}FA/\sim$, see \cite{tens}, and the relation $\tau^*$ is defined by
	$$(x/\sim, y/\sim)\in \tau^*\Leftrightarrow\  \tau_A(x)=y$$  in which $x$ is mapping to $x/\sim$ by the colimit injection.
	It is worth noting that $\overrightarrow{{\rm lim}}_{A\in { \mathcal{P}_{\rm f}(\mathbb{D})}}FA$ together with the action 
	$\cdot:~{\rm Perm} (\mathbb{D})\times~ \overrightarrow{{\rm lim}}_{A\in { \mathcal{P}_{\rm f}(\mathbb{D})}}FA\to ~\overrightarrow{{\rm lim}}_{A\in { \mathcal{P}_{\rm f}(\mathbb{D})}}FA$ mapping each $(\pi,x/\sim)$ to $F(\pi|_A)( x)/\sim$ is a nominal set, see \cite[lemma 6.7]{Pitts2}. Also  naturality of $\tau$ indicates that $\tau^*$ is well-defined and functoriality of $F$ implies that $\tau^*$ is equivariant. To prove that ${\mathcal{P}_{\rm fs}}^*$ is a left adjoint for ${\mathcal{P}_{\rm fs *}}$, we give the natural transformation $\eta_F:F\to{\mathcal{P}_{\rm fs *}}{\mathcal{P}_{\rm fs}}^*(F)$ to be ${\eta_F}_B:FB\to \mathcal{P}_{\rm fs *}\overrightarrow{\rm lim}_{A\in { \mathcal{P}_{\rm f}(\mathbb{D})}}FA(B)$, mapping each $x\in FB$ to $\{x/\sim\}$, in each level $B\in \mathcal{P}_{\rm f} (\mathbb{D})$, for every functor $F\in {\bf Set}^{ \mathcal{P}_{\rm f}(\mathbb{D})}$. Notice that, by the definition of action of $\overrightarrow{\rm lim}_{A\in { \mathcal{P}_{\rm f}(\mathbb{D})}}FA$, ${\rm supp} \{x/\sim\}\subseteq B$, for every $x\in FB$. Indeed, if $\pi|_B=id_B$, then $F(\pi|_B)=id_{FB}$ and hence $F(\pi|_B)(x)=x$. Also, for every $x\in FB$ and for the inclusion function $i:B\hookrightarrow C$ we have: 
	\begin{align*}
		{\mathcal{P}_{\rm fs *}}(i){\eta_F}_B(x)&=\overrightarrow{Fi}(\{x/\sim\})\\
		&=\{Fi(x)/\sim\}\\
		&={\eta_F}_C(Fi(x)).
	\end{align*}
	This indicates the naturality of $\eta_F$. Now we show that $\eta_F$ is a universal arrow, for every functor $F\in {\bf Set}^{ \mathcal{P}_{\rm f}(\mathbb{D})}$. To do so, let $\iota:F\to \mathcal{P}_{\rm fs *}Y$ be a natural transformation, for some $y\in \bf Rel(\bf Nom)$. Then there exists 
	$$\overline{\iota}:=\{(x/\sim, y)\ :\ \iota_A(x)=y\}$$
	in which $x\in FA$ maps to $x/\sim$ by the colimit injection. Naturality of $\iota$ implies that $\overline{\iota}$ is well-defined and functoriality of $F$ implies that $\overline{\iota}$ is equivariant. Also we have 
	\begin{align*}
		\mathcal{P}_{\rm fs *}(\overline{\iota})_A\circ {\eta_F}_B(x)&=\mathcal{P}_{\rm fs *}(\overline{\iota})_A(\{x/\sim\})\\
		&=\overrightarrow{\overline{\iota}}(\{x/\sim\})\\
		&=y\\
		&=\iota_A(x).
	\end{align*}
	One can easily check the uniqueness of $\overline{\iota}$ with $\mathcal{P}_{\rm fs *}(\overline{\iota})_A\circ {\eta_F}_B(x)=\iota_A(x)$.
\end{proof}  

The following example shows that the functor $\mathcal{P}_{\rm fs *}:{\bf Rel}({\bf Nom})\to {\bf Set}^{ \mathcal{P}_{\rm f}(\mathbb{D})} $ is not faithful.

\begin{example}\label{not faithful}
Suppose $R,R' \in \mathcal{R}(\mathbb{D}, \mathbb{D}^{(2)})$ with $R=\{(d,(d,d')): d\neq d' \}$ and $R'=\{(d',(d,d')): d\neq d' \}$. It is clear that $R,R'\in {\bf Rel}({\bf Nom})$ and  ${R_*}_A={R'_*}_A=\mathbb{D}^{(2)}$, for every $A\in\mathcal{P}_{\rm f}(\mathbb{D}) $, but $R\neq R'$. So the functor $\mathcal{P}_{\rm fs *}$ is not faithful.
\end{example}

\begin{remark}\label{topos}
	It is worth noting that $\emptyset$ is a zero object in {\bf Rel}({\bf Nom}), that is, both initial and terminal. Now, since the only toposes with a zero object are ones equivalent to the trivial, that is one-object-one-morphism, and the category {\bf Rel}({\bf Nom}) is patently not equivalent to that, {\bf Rel}({\bf Nom}) is not a topos. 
\end{remark}

 % % % % % % % % % % % % 
 
 \section{Natural deterministic morphism in ${\rm\bf Rel}$({\bf Nom})}

In the category ${\rm\bf Rel}$({\bf Nom}) there are several types of morphisms, each with their advantages. Each of these types can be used in combination to gain a deeper understanding of the underlying  objects. This section is devoted to an important kind of these morphisms which is called natural deterministic morphism.

\sloppypar\noindent\begin{notation} 
Let $X$ be nominal sets, and $A\in\mathcal{P}_{\rm fs}(X)$.
 The notation  $\mathcal{R}_{\rm fs}(A,B)$, in this section, refers to the set of all finitely supported relations from $X$ to $A$.
\end{notation}

\begin{definition}\label{fs functor epsilon}
Let $X$ and $Y$ be two nominal sets. A pair $(R, \sigma)$ is called a \emph{natural deterministic morphism}
if $R\in \mathcal{R}_{_{\rm fs}}(X,Y)$ and $\sigma: \epsilon_Y \circ \overrightarrow{R}\to \epsilon_X$ is a {\emph natural transformation}, in which, for every nominal set $X$,  $\epsilon_X$ is a functor from a subcategory $\mathcal{T}(X)$ of $(\mathcal{P}_{_{\rm fs}}(X), \subseteq)$  to the category  $\bf{Nom}$ defined by the following diagram, for every $A,B\in\mathcal{T}(X)$.

$$\xymatrix{
	&A \ar@{^{(}->}[d]\ar@{~>}[r]&\epsilon_X(A)=\mathcal{R}_{\rm fs}(X,A )\ar@{^{(}->}[d]\\
	&B\ar@{~>}[r]&\epsilon_X(B)=\mathcal{R}_{\rm fs}(X,B)}
$$
\end{definition}

\begin{note}
In this paper, we  either assume $\mathcal{T}(X)=(\mathcal{P}_{_{\rm fs}}(X), \subseteq)$ and consider the functor  $\epsilon_X:~\mathcal{P}_{_{\rm fs}}(X)\to \bf{Nom}$, or suppose $\mathcal{T}(X)$ to be the set of all equivariant subsets of $X$, denoted by ${\rm Eqsub}(X)$. In the latest case we denote $\epsilon_X$ by $\epsilon^{eq}_{_{X}}$ for emphasis.
\end{note}

\begin{remark}
 (i) By Proposition \ref{R_* and R^*}, every finitely supported relation $R:X\longrightarrow Y$ implies two functors $$\overrightarrow{R}:(\mathcal{P}_{_{\rm fs}}(X), \subseteq)\longrightarrow (\mathcal{P}_{_{\rm fs}}(Y), \subseteq), \ \ \
\overleftarrow{R}:(\mathcal{P}_{_{\rm fs}}(Y), \subseteq)\longrightarrow (\mathcal{P}_{_{\rm fs}}(X), \subseteq).$$
\medskip

(ii) For every $S\in {\rm Eqsub}(X)$ and $\rho \in \mathcal{R}_{\rm fs}(X, S)$, since ${\rm supp}\,(X, S)={\rm supp}\,X\cup {\rm supp}\,S $ and ${\rm supp}\,X={\rm supp}\,S=\emptyset$, empty set supports $\rho$ and we have $\pi \rho=\rho$, for every $\pi \in {\rm Perm}(\mathbb{D})$. 
\end{remark}

\begin{proposition}\label{prop define deterministic}
Each equivariant $R\in \mathcal{R}(X, Y)$ determines a  natural deterministic morphism.
\end{proposition}
\begin{proof}
Define the natural transformation $({(\sigma_{R})}_{_S})_{S\in \mathcal{P}_{_{\rm fs}}(X)}$, in which ${(\sigma_{R})}_{_S}: \epsilon_Y(\overrightarrow{R}(S))\to \epsilon_X(S)$ assigns every $\rho\in \mathcal{R}_{\rm fs}(Y,\overrightarrow{R}(S))$ to ${(\sigma_{R})}_{_S}(\rho)$, defined by
 $$(x,s)\in {(\sigma_{R})}_{_S}(\rho)\Leftrightarrow\text{ there exist } y_1\in Y, y_2\in \overrightarrow{R}(S)\text{ such that }\vcenter{\xymatrix @-1.5pc @ur { x\ar[d]_{R}&s\ar[d]^{R}\\
  y_1\ar[r]_{\rho} &y_2.}}
$$
It is clear ${(\sigma_{R})}_{_S}$'s are maps.  The naturality of $\sigma_R$ is obtained easily.
\end{proof}

\begin{proposition}
Let $X$ be a nominal set and $R\in \mathcal{R}(X, X)$ be an equivariant injective relation. Then, ${(\sigma_{R})}_{_X}(\sharp_{_X})\subseteq \sharp_{_X}$.
\end{proposition}

\begin{proof}
Applying Proposition \ref{prop define deterministic} we have:
$$(x,y)\in {(\sigma_{R})}_{_X}(\sharp_{_X})\Leftrightarrow\text{ there exist } y_{1}\in X, y_{2}\in \overrightarrow{R}(X)\text{ such that }\vcenter{\xymatrix @-1.5pc @ur { x\ar[d]_{{R}}&y\ar[d]^{{R}}\\
  y_{1}\ar[r]_{\sharp_{_X}} & y_{2}.}}
$$
Since $R$ is injective, by Corollary \ref{yhm}(ii), ${\rm supp}\, x\subseteq {\rm supp}\,y_1$ and  ${\rm supp}\, y\subseteq {\rm supp}\,y_2$. Now, since ${\rm supp}\,x \cap {\rm supp}\,y\subseteq {\rm supp}\,y_1\cap {\rm supp}\,y_2$ and $(y_1, y_2)\in \sharp_{_X}$, we get $(x,y)\in \sharp_{_X}$.
\end{proof}

\begin{theorem}\label{preserve}
Suppose $X,Y$ are nominal sets and  $R\in \mathcal{R}(X, Y)$ is equivariant. 

{\rm (i)} Then ${(\sigma_{R})}_{_S}$ is order-preserving, for every $S\in \mathcal{P}_{\rm fs}(X)$.

\medskip

{\rm (ii)} If $R,T\in \mathcal{R}(X,Y) $ are equivariant relations and $R\subseteq T$, then ${(\sigma_{R})}_{_S}(\rho)\subseteq{(\sigma_{T})}_{_S}(\rho)$, for every $S\in \mathcal{P}_{\rm fs}(X)$.
\end{theorem}

\begin{proof}
{\rm (i)} Suppose $\rho _{1},\rho _{2}\in  \mathcal{R}(Y, \overrightarrow{R}(X))$ where $\rho _{1}\subseteq \rho _{2}$. Let $(x,y)\in{(\sigma_{R})}_{_S}(\rho_{1} )$. So we have:
\begin{align*}
(x,y)\in {(\sigma_{R})}_{_S}(\rho _{1} )&\Leftrightarrow\text{ there exist } y_{1}\in Y, y_{2}\in \overrightarrow{R}(S)\text{ such that }\vcenter{\xymatrix @-1.5pc @ur { x\ar[d]_{R}&y\ar[d]^{R}\\
  y_{1}\ar[r]_{\rho _{1}} & y_{2}}}\\
&\Rightarrow \text{ there exist }  y_{1}\in Y, y_{2}\in \overrightarrow{R}(S)\text{ such that }\vcenter{\xymatrix @-1.5pc @ur { x\ar[d]_{R}&y\ar[d]^{R}\\
  y_{1}\ar[r]_{\rho _{2}} & y_{2}}}.\\ \tag*{\qedhere}
\end{align*}

{\rm (ii)} For every $S\in \mathcal{P}_{\rm fs}(X)$ and $\rho\in  \epsilon_Y(\overrightarrow{R}(S))$ we have:
 \begin{align*}
(x,s)\in {(\sigma_{R})}_{_S}(\rho)&\Leftrightarrow\text{ there exist } y_{1}\in Y, y_{2}\in \overrightarrow{R}(S)\text{ such that }\vcenter{\xymatrix @-1.5pc @ur { x\ar[d]_{R}&s\ar[d]^{R}\\
  y_{1}\ar[r]_{\rho} & y_2}}\\
  &\Rightarrow\text{ there exist } y_1\in Y, y_2\in \overrightarrow{R}(T)\text{ such that }\vcenter{\xymatrix @-1.5pc @ur { x\ar[d]_{T}&s\ar[d]^{ T}\\
  y_1\ar[r]_{\rho} & y_2}}\\
  &\Leftrightarrow (x,s)\in {(\sigma_{T})}_{_S}(\rho).\tag*{\qedhere}
\end{align*}
 \end{proof}
 
 \medskip
 
\subsection{The properties of natural deterministic morphism}

\begin{proposition}\label{mohemm}
Suppose $X$ and $Y$ are  nominal sets and $\rho\in \mathcal{R}_{\rm fs}(X, Y)$. If $S\in \mathcal{P}_{_{\rm fs}}(X)$ and $\overrightarrow{R}(S)\neq \emptyset$, then

{\rm(i)}  ${\rm supp}\,{(\sigma_{R})}_{_S}(\rho) \subseteq {\rm supp}\,{(\sigma_{R})}_{S} \cup {\rm supp}\,\rho$.

{\rm(ii)}  ${\rm supp}\,{(\sigma_{R})}_{_S}\subseteq {\rm supp}\,R \cup {\rm supp}\,S$.
\end{proposition}

\begin{proof}
{\rm(i)} Since ${(\sigma_{R})}_{_S}$'s are maps, by Lemma \ref{suppor of fs map}, we have ${\rm supp}\,{(\sigma_{R})}_{_S}(\rho) \subseteq {\rm supp}\,{(\sigma_{R})}_{_S} \cup {\rm supp}\,\rho$. 

{\rm(ii)} Let $d_1, d_2\notin {\rm supp}\,R \cup {\rm supp}\,S$. Then, $(d_1\ d_2) S=S$ and $\overrightarrow{R}((d_1\ d_2)x)=(d_1\ d_2) \overrightarrow{R}(x)$. We show that $(d_1\ d_2) {(\sigma_{R})}_{_S} (\rho)={(\sigma_{R})}_{S}((d_1\ d_2)\rho)$. To do so, let $(x, s)\in {(\sigma_{R})}_{_S} ((d_1\ d_2)\rho)$. Then, there exist $y_1\in Y$, $y_2\in \overrightarrow{R}(S)$ with $(x, y_1), (s, y_2)\in R$ and $(y_1, y_2)\in (d_1\ d_2) \rho$. So, $((d_1\ d_2) y_1, (d_1\ d_2) y_2)\in \rho$. Since $(d_1\ d_2) S=S$ and $\overrightarrow{R}((d_1\ d_2)x)=(d_1\ d_2) \overrightarrow{R}(x)$, we have $(d_1\ d_2)s\in S$ and $((d_1\ d_2) x, (d_1\ d_2)y_1)\in R$. Thus, $((d_1\ d_2) x, (d_1\ d_2)s)\in {(\sigma_{R})}_{_S}(\rho)$ and so $(x, s)\in (d_1\ d_2) {(\sigma_{R})}_{_S}(\rho)$. The other side is proved similarly.
\end{proof}

\begin{note}\label{exmple deter morphism}
	Suppose $X,Y$ are nominal sets. Then

	{\rm(i)} if $R\in \mathcal{R}_{\rm fs}(X, Y)$, then  ${(\sigma_{R})}_{_S}$'s are finitely supported.
	\medskip
	
	{\rm(ii)} if $R\in \mathcal{R}(X, X)$ is equivariant, then ${(\sigma_{R})}_{_X}$ and ${(\sigma_{R})}_{_X}(\Delta_{X} )$ are equivariant too.
	\end{note}

\begin{remark}\label{oploko}
  Suppose $X$ is a nominal set and $R\in \mathcal{R}(X, X)$ is equivariant. Then 
  
  {\rm(i)} ${\rm Dom} {(\sigma_{R})}_{_X}(R )\subseteq{\rm Dom}R$.
  
  {\rm (ii)} ${\rm Im} {(\sigma_{R})}_{_X}(R)\subseteq{\rm Dom}R$.
 \end{remark}

 \begin{proof}
{\rm(i)} Suppose $x\in {\rm Dom} {(\sigma_{R})}_{_X}(R )$, so there exists $y\in X$, such that $(x,y)\in {(\sigma_{R})}_{_X}(R )$. Then we have:
 $$(x,y)\in {(\sigma_{R})}_{_X}(R )\Leftrightarrow\text{ there exist } y_{1}\in Y, y_{2}\in \overrightarrow{R}(X)\text{ such that }\vcenter{\xymatrix @-1.5pc @ur { x\ar[d]_{R}&y\ar[d]^{R}\\
  y_{1}\ar[r]_{R} & y_{2}}}.$$
  Therefore $x\in {\rm Dom}R$ and ${\rm Dom} {(\sigma_{R})}_{_X}(R)\subseteq{\rm Dom}R$.
  
  {\rm (ii)} Suppose $y\in {\rm Im} {(\sigma_{R})}_{_X}(R )$, so there exists $x\in X$, such that $(x,y)\in {(\sigma_{R})}_{_X}(R )$. Then we have:
 $$(x,y)\in {(\sigma_{R})}_{_X}(R )\Leftrightarrow\text{ there exist } y_{1}\in Y, y_{2}\in \overrightarrow{R}(X) \text{ such that }\vcenter{\xymatrix @-1.5pc @ur { x\ar[d]_{R}&y\ar[d]^{R}\\
  y_{1}\ar[r]_{R} & y_{2}}}.$$
  Therefore $y\in {\rm Dom}R$ and ${\rm Im} {(\sigma_{R})}_{_X}(R )\subseteq{\rm Dom}R$.
 \end{proof}

\begin{proposition}\label{iiin} 
 Suppose $X$ is a nominal set and $R\in \mathcal{R}(X, X)$ is  equivariant. If  $R$ is  injective, then
 
 {\rm (i)} ${(\sigma_{R})}_{_X}(R)\subseteq R$.
 
 {\rm (ii)} the relation ${(\sigma_{R})}_{_X}(R)$ is  injective.

 \end{proposition}
 
 \begin{proof}
 
 {\rm (i)} Suppose $R$ is  injective and $(x,y)\in {(\sigma_{R})}_{_X}(R)$. So we have: 
$$(x,y)\in {(\sigma_{R})}_{_X}(R )\Leftrightarrow\text{ there exist } x_{1}\in X, x_{2}\in \overrightarrow{R}(X)\text{ such that }\vcenter{\xymatrix @-1.5pc @ur { x\ar[d]_{R}&y\ar[d]^{R}\\
  x_{1}\ar[r]_{R} & x_{2}.}}
$$
Since $R$ is injective, so $y=x_{1}$. Then $(x,y)\in R$ and 
${(\sigma_{R})}_{_X}(R)\subseteq R$.

{\rm (ii)} It follows from part (i).
 \end{proof}

 \begin{corollary}\label{uniform}
 Suppose $X$ is a nominal set  and $R\in \mathcal{R}(X, X)$  is an equivariant injective relation. If $(x,y)\in{(\sigma_{R})}_{_X}(R)$, then  ${\rm supp}\,x\subseteq{\rm supp}\,y$.

 \end{corollary}
 
 \begin{proof}
 Since $R$  is an equivariant injective relation, by Proposition \ref{iiin}(ii), ${(\sigma_{R})}_{_X}(R)$ is injective. Using  Corollay \ref{yhm}(ii),  ${\rm supp}\,x\subseteq{\rm supp}\,y$.
\end{proof}

 \begin{proposition}\label{tgb}
 Let $X$ be a nominal set and $R\in \mathcal{R}(X, X)$ be equivariant. Then $R$ is injective if and only if ${(\sigma_{R})}_{_X}(\Delta_{_{X}})\subseteq \Delta_{_{X}}$.
 \end{proposition}
 
 \begin{proof}
 Suppose $R$ is injective and $(x,s)\in {(\sigma_{R})}_{_X}(\Delta_{X})$. So we have:
$$(x,s)\in {(\sigma_{R})}_{_X}(\Delta_{X} )\Leftrightarrow\text{ there exists } x_{1}\in  \overrightarrow{R}(X)\text{ such that }\vcenter{\xymatrix @-1.5pc @ur { x\ar[d]_{R}&s\ar[d]^{R}\\
  x_{1}\ar[r]_{\Delta_{X}} & x_{1}.}}
$$
Since $R$ is injective, so $x=s$ and we have ${(\sigma_{R})}_{_X}(\Delta_{X} )\subseteq \Delta_{X}$.  Conversely, suppose $(x,s),(x',s)\in R$, so we have
$\vcenter{\xymatrix @-1.5pc @ur { x\ar[d]_{R}&x'\ar[d]^{R}\\
  s\ar[r]_{\Delta_{X}} & s.}}
$
Then $(x,x')\in {(\sigma_{R})}_{_X}(\Delta_{X} )$. Since ${(\sigma_{R})}_{_X}(\Delta_{X} )\subseteq\Delta_{X}$, so $x=x'$. Thus $R$ is injective.
 \end{proof}

 \begin{proposition}\label{ysllmm}
 Suppose $X$ is a nominal set and $R\in \mathcal{R}(X, X)$ is an equivariant  symmetric relation. If   ${(\sigma_{R})}_{_X}(R)$ is  well-defined, then  
 
 {\rm (i)} the relation $R$ is injective.
 
  {\rm (ii)} the relation $R$ is  well-defined.
 \end{proposition}

 \begin{proof}
 {\rm (i)} Suppose $(x',x), (x'',x)\in R$, so we have  $\vcenter{\xymatrix @-1.5pc @ur { x'\ar[d]_{R}&x\ar[d]^{R}\\
 x\ar[r]_{R} & x'}}$ and $\vcenter{\xymatrix @-1.5pc @ur { x\ar[d]_{R}&x'\ar[d]^{R}\\
 x'\ar[r]_{R} & x}}$ and $\vcenter{\xymatrix @-1.5pc @ur { x''\ar[d]_{R}&x\ar[d]^{R}\\
 x\ar[r]_{R} & x''}}$ and $\vcenter{\xymatrix @-1.5pc @ur {x\ar[d]_{R}& x''\ar[d]^{R}\\
  x''\ar[r]_{R} & x}}$. Then $(x',x), (x,x'),(x'',x), (x,x'')\in {(\sigma_{R})}_{_X}(R)$. Since ${(\sigma_{R})}_{_X}(R)$ is  well-defined, so $x'=x''$. Thus $R$ is injective.
  
  \medskip
  
  {\rm (ii)} The proof is similar to part (i).
 \end{proof}

 \begin{corollary}
 Suppose $X$ is a nominal set and $R\in \mathcal{R}(X, X)$ is an equivariant  symmetric relation. 
 
 {\rm (i)} Then ${(\sigma_{R})}_{_X}(R)$ is  well-defined if and only if $R$ is injective.
 
 \medskip
 
 {\rm (ii)}   If ${(\sigma_{R})}_{_X}(R)$ is  well-defined and $(x,y)\in R$, then ${\rm supp}\,x={\rm supp}\,y$.
 \end{corollary}
 
 \begin{proof}
{\rm (i)} Suppose $R$ is  symmetric and ${(\sigma_{R})}_{_X}(R)$ is  well-defined. Then, by Proposition \ref{ysllmm}(i),  $R$ is injective. Conversely, suppose $R$ is injective. Since $R$ is  symmetric, the relation $R$ is  well-defined.  Now, Proposition  \ref{iiin}(i) implies that ${(\sigma_{R})}_{_X}(R)$ is  well-defined.
 
 {\rm (ii)} Suppose  $R$ is an equivariant  symmetric relation and ${(\sigma_{R})}_{_X}(R)$ is  well-defined and $(x,y)\in R$. By Proposition \ref{ysllmm}(i), $R$ is injective. Then Proposition \ref {faroh}(i) implies that ${\rm supp}\,x={\rm supp}\,y$.
\end{proof}

 \begin{proposition}\label{pjsxc}
 Suppose $X$ is a nominal set and $R\in \mathcal{R}(X, X)$ is equivariant.  If $\Delta_{X} \subseteq R$, then $\Delta_{X} \subseteq {(\sigma_{R})}_{_X}(R)$.
  \end{proposition}
  \begin{proof}
 Since $(x,x)\in R$, for every $x\in X$, we have $\vcenter{\xymatrix @-1.5pc @ur { x\ar[d]_{R}&x\ar[d]^{R}\\
  x\ar[r]_{R} & x}}.$
  Thus we have $(x,x)\in~ {(\sigma_{R})}_{_X}(R)$,  for every $x\in X$, and we get the desired result.
 \end{proof}

\begin{corollary}\label{iff}
Suppose $X$ is a nominal set and $R\in \mathcal{R}(X, X)$ is equivariant. Then $R$ is total injective if and only if $  {(\sigma_{R})}_{_X}(\Delta_{X} )=\Delta_{X}$.
\end{corollary}

\begin{proof}
Suppose $R$ is total injective. Then,  ${(\sigma_{R})}_{_X}(\Delta_{X} )\subseteq\Delta_{X}$  follows from  Proposition \ref{tgb}. Since ${\rm Dom}R=X$,  we have $\Delta_{X} \subseteq{(\sigma_{R})}_{_X}(\Delta_{X} )$. Therefore $  {(\sigma_{R})}_{_X}(\Delta_{X} )=\Delta_{X}$. Conversely, suppose $  {(\sigma_{R})}_{_X}(\Delta_{X} )=\Delta_{X}$. Since  $ {(\sigma_{R})}_{_X}(\Delta_{X} )\subseteq\Delta_{X}$, applying Proposition \ref{tgb}, we have $R$ is injective. Since $\Delta_{X} \subseteq{(\sigma_{R})}_{_X}(\Delta_{X} )$, so  ${\rm Dom}R=X$. Therefore $R$ is total injective.
\end{proof}

\begin{theorem}\label{tgbd}
Let $X$ be a nominal set, $R\in \mathcal{R}(X, X)$ be  equivariant.

{\rm (i)}  If  $\rho\in \mathcal{R}_{_{\rm fs}}(X, X)$ is coreflexive, then $ {(\sigma_{R})}_{_X}(\rho )$ is symmetric.

{\rm (ii)} If $R$ is reflexive, then ${(\sigma_{R})}_{_X}(R)\subseteq R$ and so ${(\sigma_{R})}_{_X}(\Delta_{X})\subseteq R$.

{\rm (iii)} If $R$ is symmetric and transitive, then ${(\sigma_{R})}_{_X}(\Delta_{X})\subseteq R$.
\end{theorem}

\begin{proof}
{\rm (i)} Let $(x,y)\in {(\sigma_{R})}_{_X}(\rho)$. Then we have:
\begin{align*}
(x,y)\in {(\sigma_{R})}_{_X}(\rho)&\Leftrightarrow\text{ there exist } x_{1}\in  X,x_{2}\in \overrightarrow{R}(X)\text{ such that }\vcenter{\xymatrix @-1.5pc @ur { x\ar[d]_{R}&y\ar[d]^{R}\\
  x_{1}\ar[r]_{\rho} & x_{2}}}.
 \end{align*}
 Since $\rho$ is coreflexive,  $ x_{1}= x_{2}$ and we have $\vcenter{\xymatrix @-1.5pc @ur { y\ar[d]_{R}&x\ar[d]^{R}\\
  x_{1}\ar[r]_{\rho} & x_{1}}}
$. So $(y,x)\in {(\sigma_{R})}_{_X}(\rho)$.

Since $R$ is reflexive, $\Delta_{_{X}}\subseteq R$. So, $x\in \overrightarrow{R}(X)$, for all $x\in X$.

{\rm(ii)} Consider an arbitrary $(x,y)\in {(\sigma_{R})}_{_X}(R)$. Then we have:
\begin{align*}
(x,y)\in {(\sigma_{R})}_{_X}(R)&\Leftrightarrow\text{ there exist } x\in X, y\in  \overrightarrow{R}(X) \text{ such that }\vcenter{\xymatrix @-1.5pc @ur { x\ar[d]_{R}&y\ar[d]^{R}\\
  x\ar[r]_{R} & y}}\\
  &\Rightarrow (x,y)\in R.
\end{align*}

Now, since ${(\sigma_{R})}_{_X}$ is a map and $\Delta_{_{X}}\subseteq R$, we have ${(\sigma_{R})}_{_X}(\Delta_{_{X}})\subseteq {(\sigma_{R})}_{_X}(R)\subseteq R$.

{\rm(iii)} For every $(x,s)\in {(\sigma_{R})}_{_X}(\Delta_{X})$, we have:
\begin{align*}
(x,s)\in {(\sigma_{R})}_{_X}(\Delta_{X})&\Leftrightarrow\text{ there exists } y_{1}\in  \overrightarrow{R}(X) \text{ such that }\vcenter{\xymatrix @-1.5pc @ur { x\ar[d]_{R}&s\ar[d]^{R}\\
  y_{1}\ar[r]_{\Delta_{X}} & y_{1}}}\\
  &\Leftrightarrow\text{ there exists }  y_{1}\in   \overrightarrow{R}(X)\text{ such that } (x,y_{1})\in R, (s,y_{1})\in R\\
 &\Leftrightarrow\text{ there exists } y_{1}\in   \overrightarrow{R}(X)\text{ such that } (x,y_{1})\in R, (y_{1},s)\in R \\
  &\Rightarrow (x,s)\in R.\tag*{\qedhere}
\end{align*}
\end{proof}

\begin{proposition}\label{ubzp}
Suppose $X$ is a nominal set with $\mathcal{Z}(X)=\emptyset$  and $R\in \mathcal{R}(X, X)$ is a partial surjective equivariant map. If $(x,y)\in {(\sigma_{R})}_{_X}(\Delta_{_X})$, then  $(x,y)\notin \sharp _{_X}$.

\end{proposition}

\begin{proof}
Suppose $(x,y)\in {(\sigma_{R})}_{_X}(\Delta_{_X})$. Then we have:
\begin{align*}
(x,y)\in {(\sigma_{R})}_{_X}(\Delta_{_X} )&\Leftrightarrow\text{ there exists } x_{1}\in   \overrightarrow{R}(X)\text{ such that }\vcenter{\xymatrix @-1.5pc @ur { x\ar[d]_{R}&y\ar[d]^{R}\\
  x_{1}\ar[r]_{\Delta_{_X}} & x_{1}}}.
 \end{align*}
 Since $R$ is a partial surjective equivariant map ,  Corollary \ref{parsur}(ii) implies ${\rm supp}\,x_{1}\subseteq {\rm supp}\,y$ and ${\rm supp}\,x_{1}\subseteq {\rm supp}\,x$. Since $\mathcal{Z}(X)=\emptyset$, we have ${\rm supp}\,x\cap {\rm supp}\,y\neq \emptyset$. Thus   $(x,y)\notin \sharp _{_X}$.
\end{proof}

\begin{corollary}\label{ygde}
Suppose $X$ and $\mathbb{D}$ are nominal sets  and $R\in \mathcal{R}(X, \mathbb{D})$ is a partial surjective equivariant map. If $(x,y)\in {(\sigma_{R})}_{_X}(\Delta_{\mathbb{D}})$, then 

{\rm (i)} $(x,y)\notin \sharp _{_X}$.

{\rm (ii)} $x,y\notin\mathcal{Z}(X)$.
\end{corollary}

\begin{proof}
 It follows from Proposition \ref{ubzp}.
\end{proof}

 \medskip

 \subsection{Some concrete examples for natural deterministic morphism}
 
Here, we try to provide a better understanding of the concept of natural deterministic morphism by giving some examples.
 
 \begin{definition}\cite{binding}\label{bin}
 Let $X$ be a nominal set. A \emph{binding operator} on $X$ is an equivariant
map $l : X \longrightarrow \mathcal{P}_{\rm f}(\mathbb{D})$. Each $l$ gives rise to a relation $\equiv_{l}$ on $X$, defined as

$x_{1} \equiv_{l} x_{2} \Leftrightarrow \text{there exists}~ \pi\in {\rm Perm}(\mathbb{D}),~\text{such that}~ \pi \sharp {\rm supp}(x_{1}) \setminus  l(x_{1})~ \text{and}~ x_{2} = \pi  x_{1}.$
\end{definition}

\begin{remark}\cite{binding}\label{lewanh}
Let $X$ be a nominal set endowed with a binding operator $l$. Then

{\rm (i)} the relation $\equiv_{l}$ is an equivariant equivalence relation. 

{\rm (ii)} if $(x,y)\in \equiv_{l}$, then ${\rm supp}\,x\setminus  {l}(x)={\rm supp}\,y\setminus  {l}(y)$.
\end{remark}

\begin{corollary}\label{zxcv}
Let $X$ be a nominal set endowed with a binding operator $l$. Then, 

{\rm (i)} $(\sigma_{{ \equiv_{l}}})_{_X}(\Delta_{X})\subseteq \equiv_{l}$.

\medskip

{\rm (ii)} There exists  $\pi\in {\rm Perm}(\mathbb{D})$ such that $\pi \sharp {\rm supp}\,x\setminus  {l}(x)$ and  $s=\pi x$, for $(x,s)\in (\sigma_{{ \equiv_{l}}})_{_X}(\Delta_{X})$.

\medskip

{\rm (iii)} ${\rm supp}\,x\setminus  {l}(x)={\rm supp}\,s\setminus  {l}(s)$, for $(x,s)\in (\sigma_{{ \equiv_{l}}})_{_X}(\Delta_{X})$.

\end{corollary}

\begin{proof}
{\rm (i)} The proof follows from Theorem \ref{tgbd}(iii).

{\rm (ii)} The proof follows from part (i) and Definition \ref{bin}.

\medskip

{\rm (iii)} The proof follows from part (ii) and Remark \ref{lewanh}(ii).
\end{proof}

\begin{proposition}\label{marrry}
Suppose $X$ is a nominal set endowed with a binding operator $l$ and $ R\in \mathcal{R}(X, X)$ is quivariant. If there exists $x\in X$ with $(x,x)\in (\sigma_{{ \equiv_{l}}})_{_X}(R)$, then $\equiv_{l}\cap R\neq\emptyset$. 
\end{proposition}

\begin{proof}
Let $(x,x)\in(\sigma_{{ \equiv_{l}}})_{_X}(R) $. Then,
$$(x,x)\in (\sigma_{{ \equiv_{l}}})_{_X}(R)\Leftrightarrow\text{ there exist } y_1\in X, y_2 \in \overrightarrow{\equiv_{l}}(X)\text{ such that }\vcenter{\xymatrix @-1.5pc @ur { x\ar[d]_{\equiv_{l}}& x\ar[d]^{\equiv_{l}}\\
  y_1\ar[r]_{ R} & y_{2}.}}
$$
\sloppypar\noindent By Definition \ref{bin}, there exist $\pi _{1}, \pi _{2}$ with $\pi_1 x=y_1 ,\ \pi_2 x=y_2$ and $\pi_1, \pi_2 \sharp {\rm supp}\,x\setminus ~l(x)$.  Also,
${\rm supp}\,x\setminus {l}(x)={\rm supp}\,y_{1}\setminus {l}(y_{1})={\rm supp}\,y_{2}\setminus {l}(y_{2})$. Thus, 
$y_2=\pi_2 x=\pi_2 \pi_1^{-1} y_1$.
\begin{align*}
\pi_{2}{\pi_{1}}^{-1}({\rm supp}\,y_{1}\setminus {l}(y_{1}))&=\pi_{2}{\pi_{1}}^{-1}({\rm supp}\,x\setminus {l}(x))\\
&=\pi_{2}({\rm supp}\,x\setminus {l}(x))\\
&={\rm supp}\,x\setminus {l}(x)\\
&={\rm supp}\,y_{1}\setminus {l}(y_{1}).
\end{align*}
Therefore $(y_{1},y_{2})\in \equiv_{l}$ and $ \equiv_{l}\cap R\neq\emptyset$.  
\end{proof}

\begin{corollary}
Suppose $X$ is a nominal set endowed with a binding operator $l$ and $ R\in \mathcal{R}(X, X)$ is quivariant. If  $\Delta_{_X}\cap(\sigma_{{ \equiv_{l}}})_{_X}(R)\neq\emptyset$, then $\equiv_{l}\cap R\neq\emptyset$. 
\end{corollary}

\begin{proof}
It follows from Proposition \ref{marrry}.
\end{proof}

\begin{proposition}
Suppose $X$ is a nominal set endowed with a binding operator $l$. If $ R=\{(x,x'):{\rm supp}\,x\setminus  {l}(x)={\rm supp}\,x'\setminus  {l}(x')\}$, then

{\rm (i)}  $\equiv_{l}\subseteq R$.

{\rm (ii)} moreover, if $R$ is injective, then ${(\sigma_{R})}_{_S}(\equiv_{l}\cap (X\times\overrightarrow{R}(S))\subseteq \equiv_{l}$. 
\end{proposition}

\begin{proof}
{\rm (i)} Let $(x,x')\in \equiv_{l}$. Then, ${\rm supp}\,x\setminus {l}(x)={\rm supp}\,x'\setminus  {l}(x')$ and so $(x,x')\in R$.

\medskip

{\rm (ii)} Let $(x,s)\in {(\sigma_{R})}_{_S}(\equiv_{l}\cap (X\times\overrightarrow{R}(S)))$. Then,
$$(x,s)\in {(\sigma_{R})}_{_S}(\equiv_{l}\cap (X\times\overrightarrow{R}(S)))\Leftrightarrow\exists y_1\in X, y_2 \in \overrightarrow{R}(S); 
\vcenter{\xymatrix @-1.5pc @ur { x\ar[d]_{R}&s\ar[d]^{R}\\
  y_1\ar[r]_{ \equiv_{l}\cap (X\times\overrightarrow{R}(S)))} & y_{2}.}}
$$
 Since $(y_{1}, y_{2})\in \equiv_{l}\cap (X\times\overrightarrow{R}(S)))$, there exists $\pi\in {\rm Perm}(\mathbb{D})$ with $y_{2}=\pi y_{1}$ and $\pi \sharp {\rm supp}\, y_{1}\setminus {l}( y_{1})$. Also, since $(x, y_1)\in R$, by definition of $R$, we have ${\rm supp}\,y_1\setminus l(y_1)={\rm supp}\,x\setminus l(x)$. Thus we have $\pi \sharp {\rm supp}\, x\setminus {l}( x)$. Since $(x, y_1)\in R$ and $R$ is equivariant, $(\pi x, y_2)=(\pi x, \pi y_1)\in R$. Now, since $(\pi x, y_2), (s, y_2)\in R$ and $R$ is injetive, we get $\pi x=s$.
Therefore, $\pi x=s$ and $\pi \sharp {\rm supp}\,x\setminus l(x)$ and so $(x,s)\in \equiv_{l}$.
\end{proof}

We recall from \cite{Pitts2} that an equivalence relation $\rho$ over a nominal set $X$ is called \emph{an equivariant equivalence relation (or congruence)} on $X$, whenever $\rho$ is equivariant as a subset of $X\times X$. We also recall that if $X$ and $Y$ are nominal sets, then $(X\times Y)/\mathord\approx$ is a nominal set where $\approx$ is a congruence on $X\times Y$ defined by 
$$(x,y) \approx (x',y') \,  \Longleftrightarrow \, \pi(x,y)=(x',y'),$$
for some $\pi \in {\rm Perm}(\mathbb{D})$ with $\pi \sharp ({\rm supp}\,y\setminus{\rm supp}\,x).$ We also have ${\rm supp}\,y\setminus{\rm supp}\,x={\rm supp}\,y'\setminus{\rm supp}\,x'$ and  the equivalence class of $(x,y)$ denoted by $(x,y)/\mathord\approx$ where ${\rm supp}\,(x,y)/\mathord\approx={\rm supp}\,y\setminus{\rm supp}\,x$.

\begin{example}
\sloppypar	For every $X\in \mathbf{Nom}$ and $i:Y\hookrightarrow Y'\in {\rm Eqsub}(X)$, the assignment $\epsilon^{^{eq}}_X~:{\rm Eqsub}(X)\to \mathbf{Nom}$ defined by 
$\epsilon^{^{eq}}_{_{X}}(Y)=(X\times Y)/\mathord\approx$ and $\epsilon^{^{eq}}_{_{X}} (i)[(x,y)/\mathord\approx]=(x,i(y))/\mathord\approx$
is a functor.
\end{example}

\begin{remark}
Suppose $X$ is a nominal set and $T\in {\rm Eqsub}(X)$. If $R\in \mathcal{R}(X, X)$ is an equivariant relation, then $\overrightarrow{R}(T)$ is a nominal set.
\end{remark}

\begin{proposition}
Let $X$ be a nominal set. Then, each injective equivariant relation $R\in \mathcal{R}(X, X)$ determines a natural deterministic morphism.
\end{proposition}
\begin{proof} 
Let $T$ be an equivariant subset of $X$ and $z\in \overrightarrow{R}(T)$. Since $R$ is injective, there exists unique $t\in T$ with $(t,z)\in R$.
Define $({(\sigma_{R})}_{_T})_{_ {T\in {\rm Eqsub}(X)}}$ in which ${(\sigma_{R})}_{_T}: \epsilon_X^{^{eq}}(\overrightarrow{R}(T)) \to \epsilon^{^{eq}}_X(T)$ assigns every $(x,z)/\mathord\approx\in (X\times \overrightarrow{R}(T))/\mathord\approx$ to $(x,t)/\approx\in (X\times T)/\mathord\approx$. We show that ${(\sigma_{R})}_{_T}$'s well-defined. To do so, let $(x,z)/\mathord\approx=(x',z')/\mathord\approx$ where $z, z'\in \overrightarrow{R}(T)$. Then, we show that there exists $\pi \in {\rm Perm}(\mathbb{D})$ with $\pi (x,t)=(x',t')$ and $\pi \sharp {\rm supp}\,t-{\rm supp}\,x$.
The assumption $(x,z)/\mathord\approx=(x',z')/\mathord\approx$ implies that there exists $\pi \in {\rm Perm}(\mathbb{D})$ with $\pi (x,z)=(x',z')$ and $\pi \sharp ({\rm supp}\,z-{\rm supp}\,x)$.
 Since $R$ is equivariant and injective, we get
 $$(t, z)\in R \Rightarrow (\pi t, z')=(\pi t, \pi z)\in R, \, \ \ (t',z'), (\pi t, z')\in R \Rightarrow \pi t=t'.$$ Thus, $\pi (x,t)=(x', t')$.
Also, since $R$ is injective, by Corollary \ref{yhm}(ii), we have ${\rm supp}\,t\subseteq {\rm supp}\,z$. Thus, ${\rm supp}\,t-{\rm supp}\,x\subseteq {\rm supp}\,z-{\rm supp}\,x$ and so
$\pi \sharp ({\rm supp}\,t-{\rm supp}\,x)$.

Now, we show that $({(\sigma_{R})}_{_T})_{_{T\in {\rm Eqsub}(X)}}$'s are equivariant. Let $\pi \in {\rm Perm}(\mathbb{D})$ and
$\pi {(\sigma_{R})}_{_T}(x,z)/\mathord\approx=(\pi x, \pi t)/\mathord\approx$. Then, we show ${(\sigma_{R})}_{_T}(\pi x, \pi z)/\mathord\approx=(\pi x, \pi t)/\mathord\approx$. Notice that, $z\in \overrightarrow{R}(T)$ and $(t, z)\in R$.
Since $\pi z\in \overrightarrow{R}(T)$ and $R$ is injective, there exists unique $t'\in T$ with $(t', \pi z)\in R$. Take ${(\sigma_{R})}_{_T}(\pi x, \pi z)/\mathord\approx=(\pi x, t')/\mathord\approx$.
On the other hand, since $(t, z)\in R$ and $R$ is equivariant, we have $(\pi t, \pi z)\in R$. Now, since $R$ is injective and $(t', \pi z), (\pi t, \pi z)\in R$, we get that $\pi t=t'$. Thus,
$\pi {(\sigma_{R})}_{_T}(x,z)/\mathord\approx=(\pi x, \pi t)/\mathord\approx=(\pi x, t')/\mathord\approx={(\sigma_{R})}_{_T}(\pi x, \pi z)/\mathord\approx$.
The naturality of $\sigma_R$ is obtained easily.
\end{proof}
\medskip

\subsection{Stochastic maps between nominal sets}

In this subsection, we introduce another morphism in the category ${\bf Rel}({\bf Nom})$, see the following definition.
\begin{definition}
Let $X$ and $Y$ be two nominal sets. A \emph{stochastic map} is a natural deterministic morphism $(f, \sigma)$ in which $f: X\longrightarrow Y$ is a finitely supported map and $\sigma: \epsilon_Y \circ f\to \epsilon_X$ is a natural transformation.

\end{definition}

It is worth noting that, by Proposition \ref{prop define deterministic}, every equivariant map $f:X\longrightarrow Y$ determines the stochastic map $(f, \sigma_f)$. In this subsection, we focus  on a stochastic map.

\begin{example}
	 Given each nominal set $X$, the support map ${\rm supp}:X\longrightarrow \mathcal{P}_{\rm f}(\mathbb{D})$ gives rise a stochastic map $({\rm supp}, \sigma_{\rm supp})$.
	\end{example}

\begin{theorem}\label{th for delta D}
Given a non-discrete nominal set $X$, then

{\rm (i)} $x\sharp y$ if and only if $(x,y)\notin {(\sigma_{\rm supp})}_{_ X}(\Delta_{\mathbb{D}})$.

{\rm (ii)} ${(\sigma_{\rm supp})}_{_ X}(\Delta_{\mathbb{D}})=X\times X\setminus \sharp_{_X}$.
\end{theorem}

\begin{proof}
Let $f={\rm supp}$. Then, applying Proposition \ref{prop define deterministic}, we have:
$$(x,y)\in {(\sigma_{f})}_{_X}(\Delta_{\mathbb{D}})\Leftrightarrow\text{ there exists } d_1\in \overrightarrow{R}(X)\text{ such that }\vcenter{\xymatrix @-1.5pc @ur { x\ar@<1ex>[d]_{f}&y\ar@<1ex>[d]^{f}\\
  d_1\ar[r]_{\Delta_{\mathbb{D}}} &d_1.}}
$$
So, $(x,y)\in {(\sigma_{f})}_{_X}(\Delta_{\mathbb{D}}) \Leftrightarrow d_1\in {\rm supp}\,x \cap {\rm supp}\,y$.

\medskip

{\rm (ii)}  We have:
 
\begin{align*}
 {(\sigma_{\rm supp})}_{_ X}(\Delta_{\mathbb{D}}) &= \{(x, y): {\rm supp}\,x \cap {\rm supp}\, y\neq \emptyset\} \\
& = \{(x,y): (x,y)\notin \sharp\}\\
& =X\times X\setminus \sharp_{_X}.\tag*{\qedhere}
\end{align*}
\end{proof}

\begin{corollary}
Let $f:X\longrightarrow X$ be  in $\mathbf{Nom}$. Then $f$ is  injective  if and only if $  {(\sigma_{f})}_{_X}(\Delta_{X} )=\Delta_{X}$.
\end{corollary}

\begin{proof}
	It follows from Corollary \ref{iff}.
\end{proof}

%	\begin{proposition}
%		Let $X,Y$ be two nominal sets and  $f:X\rightharpoonup Y$ be an equivariant and well-defined relation and  $\rho\in \mathcal{R}_{\rm fs}(Y,Y)$ be reflexive  and  $A\in \mathcal{P}_{\rm fs}(X)$. Then $ f\vert_{_A}$ is an  equivariant map if and only if $\Delta_{A} \subseteq\sigma_{{f\vert_{_{_A}}}{_{A}}}(\rho)$.
%	\end{proposition}
	
%	\begin{proof}
%		Suppose $ f\vert_{_A}$ is an  equivariant map  and $x\in A$. Since ${\rm Dom} f\vert_{_A}=A$, then there exists  $y\in Y$, such that $(x,y)\in f\vert_{_A}$. So we have
%		$\vcenter{\xymatrix @-1.5pc @ur { x\ar@^{->}[d]_{f\vert_{_A}}&x\ar@^{->}[d]^{f\vert_{_A}}\\
%				y\ar@^{->}[r]_{\Delta_{Y}\subseteq\rho} & y.}}$ Then $(x,x)\in \sigma_{{f\vert_{_{_A}}}{_{A}}}(\rho)$, for every $x\in A$. Conversely, suppose $f\vert_{_A}$ is well-defined and $\Delta_{A} \subseteq\sigma_{{f\vert_{_{_A}}}{_{A}}}(\rho)$.  Since ${\rm Dom} f\vert_{_A}\subseteq A$, so we want to show that $A\subseteq {\rm Dom} f\vert_{_A}$. Now suppose  $x\in A$,  since $\Delta_{A} \subseteq\sigma_{{f\vert_{_{_A}}}{_{A}}}(\rho)$, so $(x,x)\in \sigma_{{f\vert_{_{_A}}}{_{A}}}(\rho)$ and we have:
%		$$(x,x)\in \sigma_{{f\vert_{_{_A}}}{_{A}}}(\rho)\Leftrightarrow\text{ there exists } y_{1} \in Y\text{ such that }\vcenter{\xymatrix @-1.5pc @ur { x\ar@^{->}[d]_{f\vert_{_A}}&x\ar@^{->}[d]^{f\vert_{_A}}\\
%				y_{1}\ar@^{->}[r]_{\Delta_{Y}\subseteq\rho} & y_{1}.}}
%		$$
%		Then $x\in {\rm Dom}f\vert_{_A}$ and  ${\rm Dom} f\vert_{_A}=A$.
%	\end{proof}

	\begin{proposition}
		Let  $f:X\longrightarrow Y$ be an equivariant surjective map between nominal sets.  If $(\sigma_f)_{_A}$ maps every reflexive relation $\rho\in \mathcal{R}_{\rm fs}(Y,Y)$ to a subset of $\Delta_{A}$, for every  $A\in \mathcal{P}_{\rm fs}(X)$, then $f$ is an isomorphism.
	\end{proposition}

	\begin{proof}
	 To show that $f$ is injective,  we note that  $(\sigma_f)_{_X}(\rho)\subseteq\Delta_{A}$.  If $(x, y), (x',y)\in   f$ then, by the diagram $\vcenter{\xymatrix @-1.5pc @ur { x\ar@<1ex>[d]_{ f}&x'\ar@<1ex>[d]^{ f}\\
				y\ar[r]_{\rho} & y}}$, we have $(x,x')\in (\sigma_f)_{_X}(\rho )\subseteq\Delta_{A}$ and $x=x'$. Therefore $f$ is an  isomorphism.
	\end{proof}

	\begin{proposition}\label{mapp1}
		Let $f:X\longrightarrow X$ be  in $\mathbf{Nom}$,  such that   $ f\vert_{_A}$ is  bijective where $A\in \mathcal{P}_{\rm fs}(X)$. Then  ${(\sigma_{{f\vert_{_{_A}}}})}_{_A}(\rho)\cap \Delta_{A}\neq\emptyset$ if and only if   $\rho\cap  \Delta_{Y}\neq\emptyset$, for every $\rho \in \mathcal{R}_{{\rm fs}}(Y,Y)$.
	\end{proposition}

	\begin{proof}
		Suppose $ {(\sigma_{{f\vert_{_{_A}}}})}_{_A}(\rho)\cap \Delta_{A}\neq\emptyset$. Then there exsits $(x,x)\in {(\sigma_{{f\vert_{_{_A}}}})}_{_A}(\rho)\cap \Delta_{A}$. So we have $\vcenter{\xymatrix @-1.5pc @ur { x\ar@<1ex>[d]_-{{f\vert_{_A}}}&x\ar@<1ex>[d]^-{{f\vert_{_A}}}\\
				f(x)\ar[r]_{\rho} & f(x).}}$
 Thus $(f(x),f(x))\in\rho\cap  \Delta_{Y}\neq\emptyset$. Conversely, suppose $(y,y)\in \rho\cap  \Delta_{Y}$. Since $ f\vert_{_A}$ is bijective, so there exists a unique $x\in A$, such that, we have $\vcenter{\xymatrix @-1.5pc @ur { x\ar@<1ex>[d]_-{{f\vert_{A}}}&x\ar@<1ex>[d]^-{{f\vert_{_A}}}\\
				y\ar[r]_{\rho} & y.}}$ Thus $(x,x)\in {(\sigma_{{f\vert_{_{_A}}}})}_{_A}(\rho)$ and $ {(\sigma_{{f\vert_{_{_A}}}})}_{_A}(\rho)\cap \Delta_{A}\neq\emptyset$.
	\end{proof}

	\begin{theorem}\label{mapp3}
		 Suppose  $f:X\longrightarrow Y$ is an isomorphism in $\bf{Nom}$. Then the stochastic morphism $(f,\sigma_f)$
		 
{\rm (i)} preserves and reflects well-defined relations.

{\rm (ii)} preserves and reflects  injective.

{\rm (iii)} preserves and reflects constant relations.
	\end{theorem}

	\begin{proof}
	{\rm (i)}	Suppose $\rho$ is well-defined and $(x,y),(x,y')\in (\sigma_f)_{_A}(\rho)$. So we have: 
		$\vcenter{\xymatrix @-1.5pc @ur { x\ar@<1ex>[d]_{f}&y\ar@<1ex>[d]^{f}\\
				f(x)\ar[r]_{\rho} &f(y)}}$ and 
		$\vcenter{\xymatrix @-1.5pc @ur { x\ar@<1ex>[d]_{f}&y'\ar@<1ex>[d]^{f}\\
				f(x)\ar[r]_{\rho} &f(y')}}$.
		Since  $\rho$ is well-defined, so  $f(y)=f(y')$. Because $f$ is injective, then $y=y'$. Thus $ (\sigma_f)_{_A}(\rho)$ is well-defined.  Conversely, suppose $ (\sigma_f)_{_A}(\rho)$ is well-defined and $(y,y'),(y,y'')\in\rho$. Since $f$ is bijective, so  we have $\vcenter{\xymatrix @-1.5pc @ur { {f^{-1}(y)}\ar@<1ex>[d]_{f}&{f^{-1}(y')}\ar@<1ex>[d]^{f}\\
				y\ar[r]_{\rho} & y'}}$ and $\vcenter{\xymatrix @-1.5pc @ur { {f^{-1}(y)}\ar@<1ex>[d]_{f}&{f^{-1}(y'')}\ar@<1ex>[d]^{f}\\
				y\ar[r]_{\rho} & y''}}$. Then  $(f^{-1}(y),f^{-1}(y')), (f^{-1}(y),f^{-1}(y''))\in(\sigma_f)_{_A}(\rho)$. Since $(\sigma_f)_{_A}(\rho)$ and $f$ are well-defined, so $f^{-1}(y')=f^{-1}(y'')$ and $y'=y''$. Thus $\rho$ is well-defined.
	
\medskip

	{\rm (ii)} Suppose $\rho$ is injective and $(x,y),(x',y)\in (\sigma_f)_{_A}(\rho)$. So we have
		$\vcenter{\xymatrix @-1.5pc @ur { x\ar@<1ex>[d]_{f}&y\ar@<1ex>[d]^{f}\\
				f(x)\ar[r]_{\rho} & f(y)}}$ and 
		$\vcenter{\xymatrix @-1.5pc @ur { x'\ar@<1ex>[d]_{f}&y\ar@<1ex>[d]^{f}\\
				f(x')\ar[r]_{\rho} & f(y).}}$ Since  $\rho$ and $f$ are injective, so $f(x)=f(x')$ and $x=x'$.  Thus $(\sigma_f)_{_A}(\rho)$ is injective. Conversely, suppose $(\sigma_f)_{_A}(\rho)$ is injective and $(y',y), (y'',y)\in\rho$. Since $f$ is bijective, so 
		we have $\vcenter{\xymatrix @-1.5pc @ur {{ f^{-1}(y')}\ar@<1ex>[d]_-{f}&{f^{-1}(y)}\ar@<1ex>[d]^-{f}\\
				y'\ar[r]_{\rho} & y}}$
		and 
		$\vcenter{\xymatrix @-1.5pc @ur {{f^{-1}(y'')}\ar@<1ex>[d]_-{f}&{f^{-1}(y)}\ar@<1ex>[d]^-{f}\\
				y''\ar[r]_{\rho} & y}}$.  Then $(f^{-1}(y'),f^{-1}(y)),(f^{-1}(y''),f^{-1}(y))\in (\sigma_f)_{_A}(\rho)$, because  $(\sigma_f)_{_A}(\rho)$ is injective, then $f^{-1}(y')=f^{-1}(y'')$. Since $f$ is well-defined, so $y'=y''$. Thus $\rho$ is injective.

	\medskip
	
{\rm (iii)} Suppose $\rho$ is  constant and $ (\sigma_f)_{_A}(\rho)$ is not constant, that is, there exist $(x,y),(x',y')\in (\sigma_f)_{_A}(\rho)$, where $y\neq y'$. So we have $\vcenter{\xymatrix @-1.5pc @ur { x\ar@<1ex>[d]_{f}&y\ar@<1ex>[d]^{f}\\
				f(x)\ar[r]_{\rho} & f(y)}}
		$
		and $\vcenter{\xymatrix @-1.5pc @ur { x'\ar@<1ex>[d]_{f}&y'\ar@<1ex>[d]^{f}\\
				f(x')\ar[r]_{\rho} & f(y').}}
		$
		Since $\rho$ is constant, so $f(y)=f(y')$. Because $f$ is injective, then $y=y'$,  this is a contradiction. Thus $(\sigma_f)_{_A}(\rho)$ is constant. Conversely, suppose $(\sigma_f)_{_A}(\rho)$ is constant and $\rho$ is not constant, that is, there exist  $(y_{1},y_{2}), (y_{3},y_{4})\in \rho$, where $y_{2}\neq y_{4}$. Since $f$ is bijective, so   we have
		$\vcenter{\xymatrix @-1.5pc @ur { {f^{-1} (y_{1})}\ar@<1ex>[d]_-{f}&{f^{-1} (y_{2})}\ar@<1ex>[d]^-{f}\\
				y_{1}\ar[r]_{\rho} & y_{2}}}
		$
		and 
		$\vcenter{\xymatrix @-1.5pc @ur { f^{-1} (y_{3})\ar@<1ex>[d]_-{f}&{f^{-1} (y_{4})}\ar@<1ex>[d]^-{f}\\
				y_{3}\ar[r]_{\rho} & y_{4}.}}
		$
		Then $(f^{-1} (y_{1}),f^{-1} (y_{2})),(f^{-1} (y_{3}),f^{-1} (y_{4}))\in(\sigma_f)_{_A}(\rho)$. Since $(\sigma_f)_{_A}(\rho)$ is constant and  $f$ is well-defined, so   $f^{-1} (y_{2})=f^{-1} (y_{4})$ and $  y_{2}=y_{4}$,  this is a contradiction. Thus $\rho$ is constant.
	\end{proof}

\begin{proposition}\label{fereshte}
	Let $f:X\longrightarrow X$ be  in $\mathbf{Nom}$. Then

{\rm (i)} the assignment $(\sigma_f)_{_X}$ preserves  reflexive  relations.

\medskip
 
{\rm (ii)} the assignment $(\sigma_f)_{_X}$ preserves symmetric relations.

\medskip

{\rm (iii)} the assignment $(\sigma_f)_{_X}$ preserves transitive  relations.
\end{proposition}

\begin{proof}
	{\rm (i)} Suppose  $\rho\in \mathcal{R}_{\rm fs}(X,f(X))$ is reflexive. So, for every $x\in X$, we have $\vcenter{\xymatrix @-1.5pc @ur { x\ar@<1ex>[d]_-{f}&x\ar@<1ex>[d]^-{f}\\
			f(x)\ar@^{->}[r]_{\rho} & f(x)}}.$ Thus $ (x,x)\in (\sigma_f)_{_X}(\rho )$.

\medskip

{\rm (ii)} Suppose  $\rho\in \mathcal{R}_{\rm fs}(X,f(X))$ is symmetric and $(x,y)\in (\sigma_f)_{_X}(\rho )$. So we have $\vcenter{\xymatrix @-1.5pc @ur { x\ar@<1ex>[d]_-{f}&y\ar@<1ex>[d]^-{f}\\
			f(x)\ar@^{->}[r]_{\rho} & f(y)}}.$ Since $\rho$ is symmetric , then we have $\vcenter{\xymatrix @-1.5pc @ur { y\ar@<1ex>[d]_-{f}&x\ar@<1ex>[d]^-{f}\\
			f(y)\ar@^{->}[r]_{\rho} & f(x)}}.$ Therefore $(y,x)\in (\sigma_f)_{_X}(\rho )$.

\medskip

{\rm (iii)}    Suppose $(x,y), (y,z)\in (\sigma_f)_{_X}(\rho )$, and $\rho$ is transitive in $\mathcal{R}_{\rm fs}(X,f(X))$, so we have
$\vcenter{\xymatrix @-1.5pc @ur {x\ar@<1ex>[d]_-{f}&y\ar@<1ex>[d]^-{f}\\
				f(x)\ar@^{->}[r]_{\rho} & f(y)}}$ and 
	$\vcenter{\xymatrix @-1.5pc @ur { y\ar@<1ex>[d]_-{f}&z\ar@<1ex>[d]^-{f}\\
				f(y)\ar@^{->}[r]_{\rho} & f(z)}}$. Then we have $\vcenter{\xymatrix @-1.5pc @ur { x\ar@<1ex>[d]_-{f}&z\ar@<1ex>[d]^-{f}\\
			f(x)\ar@^{->}[r]_{\rho} & f(z)}}.$ Thus $ (x,z)\in (\sigma_f)_{_X}(\rho )$.{\qedhere}
\end{proof}

\begin{corollary}\label{equivalence2}
	Let $f:X\longrightarrow X$ be  in $\mathbf{Nom}$. 
	
	{\rm (i)} Then the stochastic morphism $(f,\sigma_f)$ preserves equivalence relations.

{\rm (ii)} If $\rho\in \mathcal{R}_{\rm fs}(X,X)$ is a congruence,  then $(\sigma_f)_{_X}(\rho )$ is a congruence.
\end{corollary}

\begin{proof}
{\rm (i)} It follows from Proposition \ref{fereshte}.

{\rm (ii)} It follows from part (i) and Proposition \ref{mohemm}.
	\end{proof}

\section{Conclusion}

The category of nominal sets and equivariant maps between them atractted a lot of interest of computer science scientists due to their unique properties. In this paper we replace equivariant relations rather than equivariant maps and consider the category ${\bf Rel}({\bf Nom})$, because  this category not only contains the category ${\bf Nom}$ and is more expressive than ${\bf Nom}$ but also because of the kind of morphisms in this category, one can allows to work various structures that are not functions. For example ${\bf Rel}({\bf Nom})$ can be used to model dependent types, which are types that depend on several values, or data that can be correlated by relations. A deterministic morphism, which  give each input data set a specific output of the same type, are also introduce in this paper. On the other hand, each input can be given a set of outputs by using  stochastic maps, which there is a given likelihood that each will occur.  We also introduce and examin stochastic maps in this paper.

\medskip

This paper consists of four sections. The need foundational concepts are covered in the first section. In the second section, we introduce the category $ {\bf Rel}({\bf Nom})$ consisting of nominal sets and equivariant relations between them and we examine some of properties of this category. In the third section, we define two functors $\mathcal{P}_{\rm fs *}$ and ${\mathcal{P}_{\rm fs}}^*$. In Theorem \ref{nbvcrei}, we  show that ${\mathcal{P}_{\rm fs}}^*\dashv \mathcal{P}_{\rm fs *}$, and hence the functor $\mathcal{P}_{\rm fs *}$ is the functor asigning each nominal set in $ {\bf Rel}({\bf Nom})$ to its sheaf representation. Finally, in section 4, we introduce deterministic and stochastic morphism. In Proposition \ref{prop define deterministic}, we show that every equivariant relation determines a natural deterministic morphism. Also, we investigate the deterministic morphism's support in Proposition \ref{mohemm} and  we can see the property of $(\sigma_f)_{_X}$, where $f$ is an equivariant map, in Proposition \ref{fereshte}.

\medskip

For further work in the future, we can focus on free, indecomposable, cyclic, injective objects in the category $ {\bf Rel}({\bf Nom})$ with stochastic and deterministic morphism. Also, we can study some categorical properties in this category, for example existence of monad, the Kleisli and Eilenberg-Moore categories, filtered category and  sheaf representation of nominal sets in the category $ {\bf Rel}({\bf Nom})$ with stochastic and deterministic morphism.

	%%%%%%%%%%%%%%%%%%%%%%%%%%%%%%%%%%%%%%%%%%%%

 \end{document}